\documentclass[11pt]{amsart}

\usepackage{amssymb,amsmath,amscd,esint}
\usepackage{mathtools}
\usepackage{bm}
\usepackage{amsfonts}
\usepackage{graphicx}
\usepackage{xcolor}
\usepackage{mathrsfs}
\usepackage{stmaryrd}
\usepackage{epsfig,color}
\usepackage{blindtext}
\usepackage{enumerate}
\usepackage{enumitem}
\usepackage{geometry}
\usepackage{color}
\usepackage[lite, alphabetic]{amsrefs}

\usepackage[colorlinks,linktocpage]{hyperref}
\hypersetup{linkcolor=[rgb]{0,0,0.715}}
\hypersetup{citecolor=[rgb]{0,0.715,0}}

\newtheorem{thm}{Theorem}[section]
\newtheorem{cor}[thm]{Corollary}
\newtheorem{prop}[thm]{Proposition}
\newtheorem{lem}[thm]{Lemma}

\newtheorem{quest}[thm]{Question}

\newtheorem{mainthm}{Theorem}

\theoremstyle{definition}
\newtheorem{defn}[thm]{Definition}

\newtheorem{exmp}[thm]{Example}

\newtheorem{notn}[thm]{Notation}

\usepackage{comment}

\theoremstyle{remark}
\newtheorem{rem}[thm]{Remark}
\newtheorem{rems}[thm]{Remarks}

\newcommand{\Ch}{{\sf Ch}}
\let\phi\varphi

\newcommand{\tr}{\mathrm{tr}}

\DeclareMathOperator{\Hess}{Hess}

\newcommand{\hess}{{\mathrm{Hess}}}

\DeclareMathOperator{\lip}{\mathrm{lip}}
\DeclareMathOperator{\Lip}{\mathrm{Lip}}
\DeclareMathOperator{\Spt}{\mathrm{Spt}}

\newcommand{\C}{\mathbb{C}}
\newcommand{\N}{\mathbb{N}}

\newcommand{\R}{\mathbb{R}}
\newcommand{\Z}{\mathbb{Z}}
\renewcommand{\subset}{\subseteq}
\newcommand{\defeq}{\mathrel{\mathop:}=}

\newcommand{\dvol}{\mathrm{dvol}}

\newcommand{\Ric}{\mathrm{Ric}}
\newcommand{\dist}{d}

\newcommand{\meas}{\mathfrak{m}}
\newcommand{\m}{\mathfrak{m}}

\newcommand{\di}{\mathop{}\!\mathrm{d}}

\DeclareMathOperator{\RCD}{RCD}
\DeclareMathOperator{\CDe}{CD}

\makeatletter
\let\c@equation\c@thm
\makeatother
\numberwithin{equation}{section}

\title[Ricci curvature and virtual abelianness in low dimensions]{Nonnegative Ricci curvature and virtual abelianness\\in dimensions less than $12$}

\author{Dimitri Navarro}
\address[Dimitri Navarro]{Department of Mathematics, University of California, Santa Cruz, CA, USA.}
\email{dnavar17@ucsc.edu}

\author{Jiayin Pan}
\address[Jiayin Pan]{Department of Mathematics, University of California, Santa Cruz, CA, USA.}
\email{jpan53@ucsc.edu}

\author{Xingyu Zhu}
\address[Xingyu Zhu]{Michigan State University, East Lansing, MI, USA.}
\email{zhuxing3@msu.edu}

\begin{document}

\begin{abstract}
   For any complete Riemannian manifold $M^n$ with nonnegative Ricci curvature and sublinear diameter growth, we establish a dimensional constraint $n\ge 4s(s-1)+k+1$ if the fundamental group $\pi_1(M)$ contains a torsion-free nilpotent subgroup of rank $k$ and step $s\ge 2$. As a consequence, if such a manifold $M$ has dimension $n<12$, then $\pi_1(M)$ is almost abelian. The proof is based on a dimensional estimate for $\RCD(0,N)$ spaces admitting $\R$-orbits of large Hausdorff dimension.
\end{abstract}


\maketitle

\tableofcontents

\parskip=4pt plus 0.5pt

\section{Introduction}\label{sec:intro}
Around 1988, Wei constructed examples of open (i.e., complete and non-compact) manifolds with positive Ricci curvature whose fundamental groups are torsion-free nilpotent \cites{Wei88,Wei_thesis}. Building on Wei's construction, Wilking later showed that any finitely generated and virtually nilpotent group can be realized as the fundamental group of an open manifold with $\Ric\ge 0$ \cite{Wilking00}. In contrast, the fundamental group of any open manifold with nonnegative sectional curvature always contains an abelian subgroup of finite index, thanks to the Cheeger--Gromoll soul theorem \cite{CG_soul}. 

Wei's original examples have sublinear diameter growth and dimension at least $30$; see \cite{Wei_thesis}*{page 16}. We remark that, by choosing suitable functions in Wei's ansatz, one can construct examples with sublinear diameter growth and dimension $13$; see Example \ref{exmp:dim13}. It is therefore natural to ask for the smallest dimension in which non-virtually-abelian fundamental groups can occur. More generally, one may ask how the nilpotent structure of $\pi_1(M)$ is constrained by the dimension of $M$.

Recall that any nilpotent group $G$ has a terminating central series
$$G=\zeta_0(G)\triangleright \zeta_1(G) \triangleright \ldots \triangleright \zeta_s(G)=\{\mathrm{id}\},$$
where $\zeta_{j+1}(G)\defeq[G,\zeta_j(G)]$. The \textit{nilpotency step},  or \textit{nilpotency class} of $G$ is the smallest integer $s$ such that $\zeta_s(G)=\{\mathrm{id}\}$. The \textit{rank} of $G$ is the sum of the rank of the abelian groups $\zeta_{j}(G)/\zeta_{j+1}(G)$.

In this paper, we prove a dimensional constraint for open $n$-manifolds $M$ with $\Ric\ge 0$ and sublinear diameter growth whose fundamental groups contain nonabelian torsion-free nilpotent subgroup of rank $k$ and step $s\ge 2$. As a consequence, $\pi_1(M)$ is virtually abelian in dimensions $<12$.

\begin{mainthm}\label{mainthm:nil_dim}
  Let $M$ be an open $n$-manifold with $\Ric\ge 0$ and sublinear diameter growth. Suppose that $\pi_1(M)$ contains a torsion-free nilpotent subgroup of rank $k$ and step $s\ge 2$, then 
  $$ n\ge 4s(s-1)+k+1.$$
  Consequently, if $n<12$, then $\pi_1(M)$ contains an abelian subgroup of finite index.
\end{mainthm}

We give several remarks on Theorem \ref{mainthm:nil_dim}.

\begin{rems}\label{rems:thmA}
   (1) The sublinear diameter growth assumption in Theorem \ref{mainthm:nil_dim} ensures that $\pi_1(M)$ is finitely generated, by the work of Sormani \cite{Sor00b}. In contrast, the fundamental group of a general open manifold with $\Ric\ge 0$ need not be finitely generated \cites{BNS_6,BNS_7}. When $\pi_1(M)$ is finitely generated, the work of Milnor \cite{Milnor68} and Gromov \cite{Gromov81} implies that $\pi_1(M)$ contains a finitely generated nilpotent subgroup of finite index; see also the work of Kapvitch--Wilking \cite{KW}. Then passing further to a subgroup of finite index, if necessarily, we can choose this nilpotent subgroup to be torsion-free.

   (2) A torsion-free nilpotent group of step $s\ge 2$ has rank at least $s+1$. Hence the first half of Theorem \ref{mainthm:nil_dim} implies that, if $\pi_1(M)$ contains a torsion-free nilpotent subgroup of step $s\ge 2$, then $n\ge 4s^2-3s+2$. In particular, for $s=2$, this gives $n\ge 12$, which yields the second half of Theorem \ref{mainthm:nil_dim}.
   
   (3) The condition $n<12$ in Theorem \ref{mainthm:nil_dim} is nearly optimal. In fact, there are examples of $13$-dimensional open manifolds with $\Ric\ge 0$, sublinear diameter growth, and $\pi_1(M)=H^3(\Z)$ the integral Heisenberg $3$-group, which is torsion-free nilpotent of rank $3$ and step $2$; see Example \ref{exmp:dim13} below. Thus only dimension $12$ remains unclear.
   
   (4) Before the present work, the best dimensional obstruction came from the growth rate of $\pi_1(M)$. By the classical theorem of Bass and Guivarc'h on the growth rate of nilpotent groups \cites{Bass72,Guiv73}, a finitely generated torsion-free nilpotent group of step $s\ge 2$ has polynomial growth of degree at least $\frac{1}{2}s(s+1)+1$. If an open $n$-manifold $M$ has $\Ric\ge 0$, then Anderson's work \cite{Anderson90}*{Theorem 1.1}, together with the at least linear volume growth of $M$ \cites{Calabi75,Yau76}, implies that the universal cover has polynomial volume growth of degree $\ge \frac{1}{2}s(s+1)+2$. By volume comparison, this gives $n\ge \frac{1}{2}s(s+1)+2$. Moreover, when the equality holds, $M$ has linear volume growth. Then by our previous work \cite{NPZ24}*{Theorem A}, $\pi_1(M)$ is virtually abelian. Hence one obtains $n\ge \frac{1}{2}s(s+1)+3$ if step $s\ge 2$. For $s=2$, this gives $n\ge 6$. 
   
   (5) Anderson also proved that if an open manifold $M^n$ has positive Ricci curvature, then $\pi_1(M)$ cannot contain a finitely generated subgroup of polynomial growth degree $\ge n-2$ \cite{Anderson90}*{Theorem 3.1(2)}. Therefore, if $M^n$ has positive Ricci curvature and $\pi_1(M)$ contains a finitely generated torsion-free nilpotent subgroup of step $s\ge 2$, then $n\ge 7$ because such a nilpotent subgroup has polynomial growth degree $\ge 4$.

   (6) If, in addition, we assume that the universal cover $\widetilde M$ has Euclidean volume growth in Theorem \ref{mainthm:nil_dim}, then no dimensional restriction is needed to obtain virtual abelianness: for any open manifold $M$ with $\Ric\ge0$ and sublinear diameter growth, if $\widetilde M$ has Euclidean volume growth, then $\pi_1(M)$ is virtually abelian. This result follows from \cite{NPZ24}*{Proposition 3.2} and \cite{Pan24}*{Theorem 1.1(2)}.
   
   (7) Besides the above-mentioned linear volume growth condition on $M$ \cite{NPZ24} and Euclidean volume growth condition on $\widetilde{M}$ \cite{Pan24}, the virtual abelianness of $\pi_1(M)$ of open manifolds with $\Ric\ge0$ has also been studied under other various geometric assumptions. See \cites{Pan21,Pan22,Pan25,Huang_dim4,Pan26,NPZ26}. 
\end{rems}

We raise two open problems.

\begin{quest}
   Is there a $12$-dimensional open manifold $M$ with $\Ric\ge 0$, sublinear diameter growth, and $\pi_1(M)\simeq H^3(\Z)$?
\end{quest}

\begin{quest}
   Does the conclusion of Theorem \ref{mainthm:nil_dim} remain true without the sublinear diameter growth condition?
\end{quest}

\begin{exmp}\label{exmp:dim13}
    We give an example of an open $13$-dimensional manifold $M$ with $\Ric\ge 0$, sublinear diameter growth, and $\pi_1(M)\simeq H^3(\Z)$ the integral Heisenberg $3$ group. The construction follows the ansatz by Wei \cite{Wei_thesis}*{page 17}. We consider a metric of the form
    $$g=dr^2 + f(r)^2 ds_{p-1}^2 +g_r,$$
    where $ds_{p-1}^2$ denotes the standard round metric on $S^{p-1}$ and $\{g_r\}_{r>0}$ is a family of Riemannian metrics on the nilmanifold $N^3=H^3(\R)/H^3(\Z)$. Following \cite{Wei_thesis}*{page 17}, we define a left-invariant metric on $H^3(\R)$ by
    $$\widetilde{g_r} = h(r)^2(dx^2 + dy^2) + g(r)^2 (dz-xdy)^2,$$
    where $x,y,z$ are the coordinates determined by
    $$\begin{pmatrix}
      1 & x & z \\
      0 & 1 & y \\
      0 & 0 & 1
    \end{pmatrix}\in H^3(\R).$$
    The metric $\widetilde{g}_r$ descends to a metric $g_r$ on $N^3$. We choose
    $$f(r)=\dfrac{r}{(1+r^2)^{1/4}},\quad h(r)=(1+r^2)^{-\alpha},\quad g(r)=\epsilon(1+r^2)^{-\beta}.$$
    with parameters
    $$\alpha=0.01,\quad \beta=0.53,\quad \epsilon=0.1,\quad p=10.$$
    Then $g$ defines a smooth Riemannian metric on a manifold $M$ diffeomorphic to $\R^p\times N^3$. It is clear that $M$ has sublinear diameter growth and dimension $p+3=13$. Finally, using the curvature formulas in \cite{Wei_thesis}*{pages 17-18}, one verifies that $(M,g)$ has positive Ricci curvature.
\end{exmp}

The proof of Theorem \ref{mainthm:nil_dim} is based on equivariant asymptotic geometry and the theory of $\RCD$ spaces. There are two main ingredients. The first is the main result of \cite{Pan26}, which relates the nilpotency step to the existence of certain $\R$-orbit large Hausdorff dimension in an equivariant asymptotic cone of the universal cover. The second is following theorem on cohomogeneity-one $\RCD(0,N)$ spaces with nilpotent group actions. It converts the existence of such an $\R$-orbit of large Hausdorff dimension into a lower bound for the synthetic dimension $N$.

\begin{mainthm}\label{mainthm:RCD_N_bdd}
   Let $H$ be a connected, simply connected, and nilpotent Lie group of $\dim(H)=k\ge 1$. Let $(Y,y,d,\meas)$ be an $\RCD(0,N)$ space, where $N>k+1$. Suppose that\\
   (1) $(Y,d,\meas)$ admits an isomorphic $H$-action, where $H$ is a closed subgroup of $\mathrm{Isom}(Y)$, such that the quotient metric space $(Y/H,\bar{y})$ is isometric to a ray $([0,\infty),0)$;\\
   (2) $H$ contains a closed $\R$-subgroup $L$ such that the orbit $Ly$ has Hausdorff dimension $s>1$.\\
   Then $N\ge 4s(s-1)+k+1$.
\end{mainthm}

\begin{exmp}\label{exmp:N_bdd_sharp}
   The lower bound $N\ge 4s(s-1)+k+1$ in Theorem \ref{mainthm:RCD_N_bdd} is sharp when $s\ge 3/2$. (For the application of Theorem \ref{mainthm:RCD_N_bdd} in this paper, only the case $s\ge 2$ is needed.) We give an example here. Let $k\ge 1$ and $\alpha=s-1\ge 1/2$. We consider the product space $\R^{k-1}\times \mathbb{G}_\alpha$, where $\mathbb{G}_\alpha$ denotes the $\alpha$-Grushin halfplane. More precisely, $\mathbb{G}_\alpha$ is the closed halfplane $\{ (r,v)\mid r\ge 0, v\in \R \}$ equipped with a weighted Riemannian metric
   $$g= \di r^2 + r^{-2\alpha} \di v^2,\quad \meas= e^{-V(r)}\dvol_g=r^p \dvol_g$$
   on the open halfplane $\{r>0\}$, where $p>0$ will be chosen later. We take the metric completion of the Riemannian distance $d_g$ and obtain a length metric $d$ on $\mathbb{G}_\alpha$. Such spaces were first studied in \cites{PanWei_Hdim,DHPW_2023} as Ricci limit spaces. The metric measure space $(G_\alpha,d,\meas)$ admits an isomorphic $\R$-action by translating in the $v$-coordinate. The $\R$-orbit at $(0,0)$ has Hausdorff dimension $1+\alpha=s$ (see \cite{DHPW_2023}*{Theorem 3.5}). By direct calculation, the open $\alpha$-Grushin plane $\mathbb{G}^+_\alpha=\{r>0\}$ satisfies
   $$\Ric_{\meas,N'}=\mathrm{Ric}+\mathrm{Hess} V-\dfrac{dV\otimes dV}{N'-2} = \dfrac{1}{r^2}\begin{pmatrix}
  -\alpha(\alpha+1)+p-\frac{p^2}{N'-2} & 0 \\
  0 & -\alpha(\alpha+1)+\alpha p
  \end{pmatrix}$$
  with respect to the orthonormal basis $\{\partial_r, r^\alpha \partial_v\}$. We set 
  $$p=2\alpha(\alpha+1),\quad N'=4\alpha(\alpha+1)+2.$$
  Because $\alpha\ge 1/2$, we have $p\ge \alpha+1$, and hence $\Ric_{\meas,N'}\ge 0$. Together with the convexity of $\mathbb{G}^+_\alpha$ in $\mathbb{G}_\alpha$, this implies that $(\mathbb{G}_\alpha,d,\meas)$ is $\RCD(0,N')$; see, for example, the arguments in \cite{RS23}*{Section 3.5}. Therefore, the product space $\R^{k-1}\times \mathbb{G}_\alpha$ is $\RCD(0,N)$ and satisfies the assumptions of Theorem \ref{mainthm:RCD_N_bdd}, with
  $$N=(k-1)+N'=4s(s-1)+k+1.$$

  We now briefly describe the proof of Theorem \ref{mainthm:RCD_N_bdd}. 
  
  The first step is to recover a weighted Riemannian metric of low regularity on the regular set of $Y$. Specifically, under assumption (1) of Theorem \ref{mainthm:RCD_N_bdd}, we show that the regular set of $Y$ is homeomorphic to $(0,\infty)\times H$ and carries a weighted Riemannian metric
  \begin{equation}\label{eq:intro_weighted_Riem}
  g=dr^2+g_r,\quad \meas=\rho(r)\di r\wedge \Omega,
  \end{equation}
  where $\{g_r\}_{r>0}$ is a family of left-invariant Riemannian metrics on $H$ that is continuous and locally $W^{1,2}$ in $r$, $\rho(r)$ is a $\CDe(0,N)$ density function, and $\Omega$ denotes a Haar measure on $H$. Then we apply the work of Mondino--Ryborz \cite{MondinoRyborz} to derive that this weighted Riemannian metric has nonnegative distributional $N$-Bakry--Emery Ricci curvature. See Theorem \ref{thm:RCD+nilpotent} for a complete description of $Y$. This part may be seen as an extension of \cite{CMS26}*{Proposition 3.7} and \cite{NPZ26}*{Section 3}, where the group $H$ is assumed to be one-dimensional. Although our proof of Theorem \ref{thm:RCD+nilpotent} follows a similar general strategy as \cite{NPZ26}*{Section 3}, several key steps require different ideas in order to handle the higher dimensional nilpotent group actions. We refer the readers to Section \ref{sec:statements_nilpotent+RCD} for an outline of the proof.

  The second step is to use the weighted Riemannian metric \eqref{eq:intro_weighted_Riem} to estimate $N$. A key input is a new curvature inequality, Proposition \ref{prop:Ric_rr}. It bounds the Bakry--Emery Ricci curvature by the ordered eigenvalues of $G(r)$, where $\{G(r)\}_{r>0}$ is a family of symmetric matrices representing the metrics $\{g_r\}_{r>0}$. Then under assumption (2) of Theorem \ref{mainthm:RCD_N_bdd}, we can estimates the smallest and largest eigenvalues of $G(r)$ as $r\to 0$ (see Section \ref{sec:eigenvalues_by_Haus_dim}). Combining the curvature inequality and the eigenvalue estimates, we prove the desired lower bound for $N$ in Section \ref{sec:prove_N_bdd}.

  \noindent\emph{Acknowledgments.} J. Pan is partially supported by the National Science Foundation DMS-2304698 and Simons Foundation Travel Support for Mathematicians. D. Navarro and X. Zhu are partially supported by the AMS-Simons Travel Grant. Most of the manuscript was written during X. Zhu's visit of the Mathematics Department at UC Santa Cruz; the authors would like to thank the department for hospitality and support.

\end{exmp}

\section{Preliminaries}\label{sec:pre}
\subsection{$\RCD$ spaces}

Our proof of Theorem \ref{mainthm:nil_dim} will rely on blow-downs of Riemannian manifolds with nonnegative Ricci curvature, and such blow-downs (equipped with a limit renormalized measure) satisfy the $\RCD(0,N)$ condition. This section is devoted to a brief introduction of $\RCD(0,N)$ spaces.

In what follows, a {metric measure space} (m.m.s.\@ for short) is a triple $(Y,d,\meas)$ such that $(Y,d)$ is a complete and separable metric space, and $\meas$ is a nonnegative $\sigma$-finite Borel measure on $Y$ with full support.

The seminal work of Lott and Villani \cite{Lott-Villani_09} and Sturm \cites{Sturm_I_06,Sturm_II_06} gave rise to the notion of $\CDe(K,N)$ spaces, which generalize Riemannian manifolds with $\Ric\ge K$ and $\dim \le N$. In this paper, we only deal with the case $K=0$, whose definition we recall below (we refer the reader to \cite{Villani_09} for an extensive discussion of $\CDe(K,N)$ spaces in general).

\begin{defn}[$\mathrm{CD}(0,N)$ spaces]
    Let $N\in (1,\infty)$. An m.m.s.\@ $(Y,d,\meas)$ satisfies the {$\mathrm{CD}(0,N)$ condition} if, given any pair of probability measures {with finite second-order moment} $\mu_0,\mu_1$ that are absolutely continuous w.r.t.\@ $\meas$, there exists a $W_2$-geodesic $\{\mu_t\}_{0\le t\le1}$ from $\mu_0$ to $\mu_1$ such that, for every $N'\ge N$, we have the following property:
        \begin{equation*}
         \mathcal{S}_{N'}(\mu_t\mid\meas)\le t\mathcal{S}_{N'}(\mu_1\mid\meas)+(1-t)\mathcal{S}_{N'}(\mu_0\mid\meas),\quad \forall 0\le t\le 1,
    \end{equation*}
    where $\mathcal{S}_{N'}(\cdot\mid\meas)$ denotes the R\'{e}nyi entropy with parameter $N'$ associated with $\meas$.
\end{defn}

Before introducing $\RCD(0,N)$ spaces, let us introduce Sobolev spaces via test plans and minimal upper gradient, following \cite{gigli_nonsmooth}. We first recall the notion of metric derivative for curves in an m.m.s.\@ $(Y,d,\meas)$.

\begin{defn}[Metric derivative, \cite{AmbrosioTilli_book}*{Section 4.1}]
    Given a curve $\gamma:[0,1]\to (Y,d)$ and $t\in[0,1]$, the metric derivative of $\gamma$ at $t$ with respect to a reference metric $d$ is defined as
$$|\dot\gamma(t)|_d=\lim\limits_{s\to 0} \dfrac{d(\gamma(t+s),\gamma(s))}{|s|}$$
whenever the limit exists. If $\gamma$ is absolutely continuous, then the limit exists $\mathcal{L}^1$-a.e.\@ on $[0,1]$.  
\end{defn}

\begin{defn}[Test plan]
    Let $e_t\colon \gamma\in C^0([0,1], Y)\to \gamma(t)\in Y$ be the time-$t$ evaluation map ($t\in [0,1]$) and let $\bm{\pi}\in \mathcal{P}(C^0([0,1],Y))$ be a probability measure on $C^0([0,1],Y)$. We say that $\bm{\pi}$ is a test plan if there exists a constant $C(\bm{\pi})>0$ such that
    \[
    (e_t)_{\sharp}\bm{\pi}\le C(\bm{\pi}) \meas, \quad \forall t\in [0,1],
    \]
    and 
    \[
    \int\int_0^1|\dot\gamma(t)|_d^2\, \di t\di\bm{\pi}(\gamma)<\infty.
    \]
    If $\gamma$ is not absolutely continuous we use the convention that $\int_0^1|\dot\gamma(t)|_d^2\, \di t=\infty$.
\end{defn}

\begin{defn}[Minimal weak upper gradient]
    Given an $\meas$-measurable function $f:Y\to \R$, an $\meas$-measurable function $G:Y\to [0,\infty]$ is called a weak upper gradient of $f$ if
    \[
    \int \left|f(\gamma(1))-f(\gamma(0))\right|\, \di\bm{\pi}(\gamma)\le \int \int_0^1G(\gamma(t))|\dot\gamma(t)|\,\di t\di\bm{\pi}(\gamma),\text{ for all test plans $\bm{\pi}$.}
    \]
    Furthermore, we say that $f$ is in the Sobolev class $S^2(Y,d,\meas)$ if it admits a weak upper gradient $G\in L^2(Y,\meas)$. We say that $G$ is a minimal weak upper gradient if $|G|\le |\tilde G|$ holds $\meas$-a.e.\@ for any weak upper gradient $\tilde G$.
\end{defn}

\begin{rem}
    The existence of a minimal weak upper gradient of a function $f\in S^2(Y,d,\meas)$ is established in \cite{ambrosio_calculus_2014}*{Definition 5.11}. We denote the minimal weak upper gradient of $f$ by $|Df|_w$.
\end{rem}

\begin{defn}[Sobolev space $W^{1,2}$]
    We define the Sobolev space $W^{1,2}(Y,d,\meas)$ as the space $ L^2(Y,\meas)\cap S^2(Y,d,\meas)$ equipped with the norm 
    \[
    \|f\|_{W^{1,2}}^2\defeq \|f\|_{L^2}^2+\||Df|_w\|_{L^2}^2,
    \]
    which turns $W^{1,2}(Y,d,\meas)$ into a Banach space. To simplify notations, we write the Sobolev space as $W^{1,2}(Y)$ when the reference metric and measure are clear from the context.
\end{defn}

\begin{defn}[Locally Sobolev space $W^{1,2}_{loc}$]\label{defn:localy_Sobolev}
   Let $\Omega\subset Y$ be an open subset. We say that $f\in L^2_{{loc}}(\Omega,\meas)$ belongs to $W^{1,2}_{loc}(\Omega,d,\meas)$ if
$\eta f \in W^{1, 2}(Y, d, \meas)$ for any $\eta \in\Lip_c(\Omega, d)$.
\end{defn}

\begin{rem}\label{rem:loc_sobolev_algebra}
    When $\meas$ is doubling (which is the case for $\CDe(0,N)$ spaces), the following properties hold:
\begin{itemize}
    \item ${L}^{\infty}_{loc}(\Omega)\cap W^{1,2}_{loc}(\Omega)$ is an algebra,
    \item if $\varphi\in W^{1,2}_{loc}(\Omega)$ satisfies $1/\varphi\in L^{\infty}_{loc}(\Omega)$, then $1/\varphi\in  W^{1,2}_{loc}(\Omega)$,
\end{itemize}
we refer the reader to  \cite{gigli_pde_2013}*{Section 2.2} for more details.
\end{rem}

While Ricci limit spaces satisfy the $\mathrm{CD}$ condition, it is also the case of certain non-Riemannian spaces such as non-Euclidean normed vector spaces and, more generally, certain non-Riemannian Finsler spaces \cite{Ohta_09}. To stay as close as possible to the Riemannian situation, Ambrosio--Gigli--Savar\'{e} introduced $\RCD(K,\infty)$ spaces in \cite{Ambrosio-Gigli-Savare_14} as $\mathrm{CD}(K,\infty)$ spaces whose Cheeger energy is a quadratic form (see also \cite{erbar_equivalence_2015} for the finite dimensional case).

\begin{defn}[Cheeger energy]
    The {Cheeger energy} $\Ch\colon L^2(Y,\meas)\to [0,\infty]$ is defined by
    \begin{equation}\label{eq:defchee}
        \Ch(f)\defeq\begin{cases}
            \frac12\int_Y |Df|_w^2\,\di\meas &\text{if $f$ admits a weak upper gradient in $L^2(Y,\meas)$}\\
            \infty &\text{otherwise}
        \end{cases}.
    \end{equation}
\end{defn}

\begin{defn}[Infinitesimal Hilbertianity \cites{Ambrosio-Gigli-Savare_14,gigli_diff_2015} and $\RCD(0,N)$ spaces]
    A metric measure space $(Y,d,\meas)$ is called infinitesimally Hilbertian if $\Ch$ is a quadratic form or, equivalently, if $W^{1,2}(Y)$ is a Hilbert space. An $\RCD(0,N)$ space ($N\in(1,\infty)$) is an infinitesimally Hilbertian $\mathrm{CD}(0,N)$ space.
\end{defn}

\begin{rem}\label{rem:Df=lip(f)}
    For $\RCD(0,N)$ spaces ($N\in(1,\infty)$), the minimal weak upper gradient construction developed by Ambrosio, Gigli, and Savaré coincides with Cheeger's construction in \cite{cheeger_differentiability_1999} (see \cite{ambrosio_calculus_2014}*{Remark 4.7}). We recall that an $\RCD(0,N)$ space $(Y,y,d,\mathfrak{m})$ is necessarily doubling and supports a local Poincaré inequality. As a result, \cite{cheeger_differentiability_1999}*{Theorem 6.1} implies that the minimal weak upper gradient of a Lipschitz function $\varphi\in\Lip(Y)$ satisfies $|D\varphi|_w = \lip(\varphi)$ $\meas$-a.e.\@, where:
    $$
    \lip(\varphi)\colon x\in Y\mapsto \limsup_{y\to x}\frac{|\varphi(y)-\varphi(x)|}{d(x,y)}\in[0,\infty),
    $$
    denotes the local Lipschitz constant of $\varphi$.
\end{rem}

\begin{rem}\label{rem:RCD_pmGH_stability}
    The class of pointed and normalized $\RCD(0,N)$ spaces is stable under pointed measured-Gromov--Hausdorff convergence (see \cite{GMS15}). In particular, for any open $N$-manifold $M$ with $\Ric\ge 0$, any asymptotic cone of $M$ with any limit renormalized measure is an $\RCD(0,N)$ space.
\end{rem}

A key property of $\RCD(0,N)$ spaces (and $\RCD(K,N)$ spaces more generally) is the existence of a well-defined dimension (see \cite{BrueSemola20}).

\begin{thm}[Rectifiable dimension]
    Let $(Y,d,\meas)$ be an $\RCD(0,N)$ space ($N\in(1,\infty)$). There exists an integer $k \le N$ such that $\meas(Y \setminus \mathcal{R}_k)=0$, where $\mathcal{R}_k$ denotes the $k$-dimensional regular set of $(Y,d,\meas)$, i.e., the set of points where $\R^k$ is the unique tangent cone. We call this integer $k$ the rectifiable dimension of $(Y,d,\meas)$.
\end{thm}

In addition, $\RCD(0,N)$ spaces support a second-order differential calculus. We recall the definition of the Laplacian and refer the reader to \cites{gigli_diff_2015,gigli_nonsmooth} for a comprehensive study.

\begin{defn}[Domain of the Laplacian $D(\Delta)$]
    Let $(Y,d,\meas)$ be an $\RCD(0,N)$ space ($N\in(1,\infty)$). A Sobolev function $\psi\in W^{1,2}(Y)$ is in the domain of the Laplacian $D(\Delta)$ if there exists $h\in L^2(Y)$ such that, for every $\phi\in \Lip_c(Y)$, we have
    $$
    \int_Y\langle \nabla \phi,\nabla \psi\rangle\di\m=-\int_Y\phi h\di\mathfrak{m},
    $$
    where $\langle \nabla \phi,\nabla \psi\rangle\defeq\frac{1}{4}\left(|D(\phi+\psi)|_w^2-|D(\phi-\psi)|_w^2\right)$. 
\end{defn}

The class of test functions will play an essential role in Section \ref{sec:loc_sobolev_metric} when deriving regularity properties of our metric.

\begin{defn}[Test functions]
    Given an $\RCD(0,N)$ space $(Y,d,\meas)$, a test function is an element $\chi\in D(\Delta)$ such that
    \[
    |\nabla\chi|\in L^{\infty}(Y)
        \quad\text{ and }\quad
        \Delta\chi\in W^{1,2}(Y).
    \]
    We denote $\mathrm{TestF}(Y,d,
    \meas)$ the set of all test functions on $(Y,d,\meas)$, and will write $\mathrm{TestF}(Y)$ when the measure and distance are clear from the context.
\end{defn}

\begin{rem}
    Test functions satisfy two important properties that will play a role in our approach (we refer the reader to \cite{gigli_nonsmooth} and the references therein for an introduction to the class of test functions).
    \begin{enumerate}[label=(\roman*), ref=(\roman*)]
        \item\label{item:test-function-Lipschitz}
        Every test function $\chi\in \mathrm{TestF}(Y)$ admits a Lipschitz representative.
    
        \item\label{item:test-function-cutoff}
        Given an open set $\Omega\subset Y$ and a set $B\subset\Omega$ satisfying
        \[
            d(B,Y\setminus\Omega)>0,
        \]
        there exists $\chi\in \mathrm{TestF}(Y)$ such that
        \[
            \chi=1
            \quad \meas\text{-a.e.\ on }B,
            \qquad
            \chi=0
            \quad \meas\text{-a.e.\ on }Y\setminus\Omega.
        \]
    \end{enumerate}
    Thanks to \ref{item:test-function-Lipschitz}, we may assume without loss of generality that test functions are Lipschitz.
\end{rem}

After establishing the regularity properties of our metric, we will rely on the results of Mondino--Ryborz \cite{MondinoRyborz} to establish explicit curvature inequalities in our setting. Below, we recall the definition of the distributional $N$-Bakry--Emery Ricci curvature and refer the reader to \cite{MondinoRyborz}*{Sections 2 and 3} and references therein for more details on metrics of low regularity.

\begin{defn}\label{defn:distributional_BE}
    Let $M$ be an $n$-dimensional smooth manifold ($n\ge 2$) equipped with
    \begin{enumerate}[label=(\alph*)]
        \item\label{item:metric}
        a Riemannian metric $g$ locally given as a positive definite matrix
        whose coefficients $g_{ij}$ and inverse coefficients $g^{ij}$ lie in
        $L^\infty_{\mathrm{loc}}\cap H^1_{\mathrm{loc}}(M)$;
    
        \item\label{item:measure_in_def_BE}
        a weighted measure
        $\meas=e^{-V}\,\mathrm{dvol}_g$, where
        $V\in C^0\cap H^1_{\mathrm{loc}}(M)$.
    \end{enumerate}
    Given smooth vector fields $X,Y$, we define $\nabla_X Y\defeq (\nabla^\flat_XY)^\sharp$, where $\nabla^\flat_XY$ is the distributional $1$-form whose pairing with a smooth vector field $Z$ satisfies 
     \begin{equation}\label{eq:Koszul}
          2\langle\nabla^\flat_X Y,Z \rangle\defeq Xg(Y,Z)+Yg(X,Z)-Zg(X,Y)-g(X,[Y,Z])-g(Y,[Z,X])+g(Z,[X,Y]).
     \end{equation}
    Given a smooth $1$-form $\theta$, the distributional covariant derivative is defined via duality as
    \begin{equation*}
        \langle\nabla_X\theta,Z\rangle\defeq X\langle\theta, Z\rangle-\langle\theta,\nabla_X Z\rangle.
    \end{equation*}
    The distributional Riemannian curvature tensor $\mathrm{Rm}$ is defined as
    \begin{equation*}
        \langle\mathrm{Rm}(X,Y)Z,\theta\rangle\defeq X\langle\nabla_YZ,\theta\rangle-Y\langle\nabla_XZ,\theta\rangle-\langle\nabla_YZ,\nabla_X\theta\rangle+\langle\nabla_XZ,\nabla_Y\theta\rangle-\langle\nabla_{[X,Y]}Z,\theta\rangle.
    \end{equation*}
    Given a local frame $\{X_1,\ldots,X_n\}$, the distributional Ricci curvature is obtained via
    \begin{equation}\label{eq:distributional_Ric}
        {\Ric(Y,Z)=g^{ij}\cdot g(\mathrm{Rm}(X_i,Y)Z,X_j)},\text{ where }g_{ij}=g(X_i,X_j).
    \end{equation}
    Given $N\in [n,\infty]$, the distributional $N$-Bakry--Emery Ricci tensor is defined as 
    \begin{equation*}
        \Ric_{\m,N}\defeq
            \begin{cases}
            \Ric+\hess V,
            &
            \text{if } N=\infty,
            \\[0.3em]
            \Ric+\hess V-\dfrac{\di V\otimes\di V}{N-n},
            &
            \text{if } n<N,
            \\[0.3em]
            \Ric,
            &\text{if } n=N, \text{ in which case we assume } V \text{ is constant}.
            \end{cases}
    \end{equation*}
    and where
    \begin{equation}\label{eq:distributional_Hess}
        \hess(V)(X,Y)=X(Y(V))-(\nabla_XY)(V)
    \end{equation}
    is the distributional Hessian of $V$.
\end{defn}

\begin{rem}\label{rem:H^1_loc}
    The space $H^{1}_{loc}(M)$ consists of functions $\varphi\colon M\to\R$ such that, for every chart $\psi\colon U\to\R^n$, the function $\varphi\circ\psi^{-1}$ lies in $W^{1,2}_{loc}(\R^n)$.
\end{rem}

We conclude our preliminaries with the following theorem which converts the $\RCD(0,N)$ condition into an inequality on the distributional $N$-Bakry--Emery Ricci curvature.

\begin{thm}[\cite{MondinoRyborz}*{Theorem 7.2}]\label{thm:CD-to-Distribution}
    Let $Y$ be a smooth manifold with boundary equipped with a geodesic distance $d$ and a nonnegative $\sigma$-finite Borel measure $\meas$ such that:
    \begin{itemize}
        \item $(Y,d)$ is the metric completion of $(Y_+,d_g)$, where $Y_+= Y\backslash \partial Y$, and $g$ is a continuous Riemannian metric on $Y_+$ with geodesic distance $d_g$ and Christoffel symbols in $L^2_{loc}(Y_+)$,
        \item $\meas=\iota_{\#}(e^{-V}\mathrm{dvol}_g)$, where $V\in C^0\cap W^{1,2}_{loc}(Y_+)$ and $\iota\colon Y_+\hookrightarrow Y$ is the inclusion map.
    \end{itemize}
    If $(Y,d,\m)$ satisfies the $\RCD(K,N)$ condition for $K\in \R$ and {$N\in[\dim(Y),\infty]$}, then the $N$-Bakry--Emery Ricci curvature $\Ric_{\meas,N}$ satisfies 
    \[
     \Ric_{\meas,N}\ge Kg,
    \]
    {on $Y_+$} (in the sense of distributions).
\end{thm}

\begin{rem}
    In comparison with \cite{MondinoRyborz}, a subtle nuance in our setting is that the Riemannian metric $g$ may not extend to the manifold boundary $\partial Y$, even if the induced distance function does. However by a standard cutoff argument, the distributional Ricci curvature lower bound holds in the interior. See \cite{NPZ26}*{section 3.6} for detailed discussion for the case $N=\infty$ (see also \cite{CMS26}*{Theorem 3.34}). When $N<\infty$, there is no modification needed since the original proof \cite{MondinoRyborz}*{Proof of Theorem 6.20} already noticed that it suffices to prove the distributional inequality in one coordinate patch.
\end{rem}

\subsection{Implication of nilpotency step in equivariant asymptotic cones}\label{sec:equiv_asym_cones}

In this section, we recall the notion of equivariant asymptotic cones of open manifolds with $\Ric\ge 0$. We also collect the main results from \cite{Pan26} on the relation between the nilpotency step of $\pi_1(M)$ and equivariant asymptotic geometry.

Let $M$ be an open manifold with $\Ric\ge 0$. For any sequence $r_i\to\infty$, after passing to a subsequence, we obtain Gromov--Hausdorff convergence
$$(r_i^{-1}M,p)\overset{GH}\longrightarrow (X,x).$$
The limit space $(X,x)$ is called an \textit{asymptotic cone} of $M$. Note that $(X,x)$ may depend on the choice of the scaling sequence $r_i^{-1}$. Equipped with a limit renormalized measure $\meas$ \cite{CCI}*{Section 1}, the metric measure space $(X,d,\meas)$ is $\RCD(0,n)$, where $n$ is the dimension of $M$ (see Remark \ref{rem:RCD_pmGH_stability}).

To study the fundamental group $\pi_1(M)$ of $M$ with $\Ric\ge 0$, it is natural to apply the equivariant Gromov--Hausdorff convergence \cites{Fukaya86,FY92} to a blow-down sequence of the universal cover. We denote $\widetilde{M}$ the Riemannian universal cover of $M$ and $\Gamma=\pi_1(M,p)$. For a sequence $r_i\to\infty$, after passing to a subsequence, we have equivariant Gromov--Hausdorff convergence
\begin{equation}\label{eq:equiv_asym_cone}
   \begin{CD}
    (r_i^{-1} \widetilde{M},\tilde{p},\Gamma) @>\mathrm{GH}>> (Y,y,G) \\
	@VV\pi V @VV \pi V\\
	(r_i^{-1} M,p) @>\mathrm{GH}>> (X,x).
    \end{CD}
\end{equation}
In \eqref{eq:equiv_asym_cone}, the limit group $G$ is a closed subgroup of the isometry group of $Y$; moreover, $(X,x)$ is isometric to the quotient metric space $(Y/G,\bar{y})$ by \cite{Fukaya86}*{Theorem 2.1}. We call $(Y,y,G)$ an \textit{equivariant asymptotic cone} of $(\widetilde{M},\Gamma)$. In general, the equivariant asymptotic cones of $(\widetilde{M},\Gamma)$ may behave wildly without uniform controls; see \cite{BNS_7}*{Section 2.2.4}.

The second-named author introduced the notion of escape rate $E(M,p)$ that measures how fast the representing geodesic loops of $\Gamma=\pi_1(M,p)$ escape from metric balls centered at $p$ \cite{Pan21}*{Definition 1.1}. By definition, $E(M,p)$ is a quantity between $0$ and $1/2$. Under the non-maximal escape rate condition $E(M,p)\not =1/2$, the equivariant asymptotic geometry of $(\widetilde{M},\Gamma)$ is well-understood thanks to the work \cite{Pan26}. For manifolds with sublinear diameter growth, they satisfy this non-maximal escape rate condition due to \cite{NPZ24}*{Proposition 3.2}.

\begin{defn}
   Let $(M,p)$ be an open $n$-manifold. We say that $(M,p)$ has sublinear diameter growth, if
   $$\limsup_{r\to\infty} \dfrac{\mathrm{diam}(\partial B_r(p))}{r}=0,$$
   where $\mathrm{diam}(\partial B_r(p))$ is measured using the extrinsic metric.
\end{defn}

\begin{prop}\label{prop:E_not_1/2}
   Let $(M,p)$ be an open $n$-manifold of $\mathrm{Ric}\ge 0$ and sublinear diameter growth, then $E(M,p)\neq1/2$.
\end{prop}

\begin{proof}
   Because $M$ has sublinear diameter growth, it has unique asymptotic cone as a ray $([0,\infty),0)$. Then the result follows directly from \cite{NPZ24}*{Proposition 3.2}.
\end{proof} 

Theorems \ref{thm:topol_dim} and \ref{thm:nil_dim} below are the main results from \cite{Pan26}.

\begin{thm}[\cite{Pan26}*{Theorem A(1) and Proposition 5.1}]\label{thm:topol_dim}
	Let $(M,p)$ be an open $n$-manifold with $\mathrm{Ric}\ge 0$. Suppose that $\Lambda=\pi_1(M,p)$ is a torsion-free nilpotent group and $E(M,p)\neq 1/2$. Then every equivariant asymptotic cone $(Y,y,G)$ of $(\widetilde{M},\Lambda)$ satisfies
    \begin{enumerate}
	\item the orbit $Gy$ has a natural connected and simply connected nilpotent group structure;
	\item any compact subgroup of $G$ fixes $y$;
	\item $G$ has at most finitely many components.
    \end{enumerate}
\end{thm}

\begin{thm}[\cite{Pan26}*{Theorem A(2)}]\label{thm:nil_dim}
	Let $(M,p)$ be an open $n$-manifold with $\mathrm{Ric}\ge 0$. Suppose that $\Lambda=\pi_1(M,p)$ is a torsion-free nilpotent group and $E(M,p)\neq 1/2$. Then there exist an equivariant asymptotic cone $(Y,y,G)$ of $(\widetilde{M},\Lambda)$ and a closed $\mathbb{R}$-subgroup $L$ of $G$ such that the orbit $Ly$ has Hausdorff dimension $\dim_H(Ly)\ge \mathrm{step}(\Lambda)$. 
\end{thm}

We note that, in general, the limit group $G$ in Theorems \ref{thm:topol_dim} and \ref{thm:nil_dim} could be abelian for every equivariant asymptotic cone $(Y,y,G)$ of $(\widetilde{M},\Lambda)$, even though $\Lambda$ itself is nilpotent and non-abelian (see the example in \cite{Pan24}*{Appendix}); in other words, information on the nilpotent step of $\Lambda$ could be lost in the limit group $G$. Theorem \ref{thm:nil_dim} shows that the nilpotency step of $\Lambda$ is reflected in the limit $G$-action of some equivariant asymptotic cone $(Y,y,G)$, in the sense that some closed $\R$-orbit at $y$ should have large Hausdorff dimension $\ge \mathrm{step}(\Lambda)$. 

\begin{rem}
   In the current form of Theorems \ref{thm:topol_dim} and \ref{thm:nil_dim}, we assumed that $\Lambda=\pi_1(M)$ is torsion-free nilpotent. We remark that this assumption is not essential. In fact, the condition $E(M,p)\neq1/2$ implies that $\pi_1(M)$ is finitely generated and contains a torsion-free nilpotent subgroup of finite index; see \cite{Pan26}*{Theorem A and Section 2.3}. Here, we stated the results in terms of torsion-free nilpotent fundamental groups for the convenience of later applications.
\end{rem}

Later in the proof, we also need the following result on equivariant Gromov-Hausdorff convergence of quotient metric spaces with quotient actions.

\begin{prop}\label{prop:eGH_quotient}
   Let $(Y_i,y_i,H_i\triangleleft G_i)$ be a sequence of pointed proper length metric spaces with isometric $H_i$ and $G_i$-actions, where both $H_i$ and $G_i$ are closed subgroups of $\mathrm{Isom}(Y_i)$ and $H_i$ is normal in $G_i$. Suppose that the sequence converges in the equivariant Gromov-Hausdorff sense
   $$(Y_i,y_i,H_i\triangleleft G_i)\overset{GH}\to (Y,y,H \triangleleft G).$$
   Then we have equivariant convergence of the quotient spaces with quotient actions
   $$(Y_i/H_i,\bar{y}_i,G_i/H_i)\overset{GH}\longrightarrow (Y/H,\bar{y},G/H).$$
\end{prop}

We include a proof of Proposition \ref{prop:eGH_quotient} in Appendix \ref{appx:conv_quotients}.

\section{$\RCD$ space with nilpotent group action and ray quotient}\label{sec:nilpotent RCD}
\subsection{Statement and outline}\label{sec:statements_nilpotent+RCD}

In this section, we prove the following theorem.

\begin{thm}\label{thm:RCD+nilpotent}
    Let $(Y,y,d,\meas)$ be an $\RCD(0,N)$ space and let $H\leq \mathrm{Iso}_{\mathrm{m.m.s}}(Y,d,\meas)$ be a closed, $k$-dimensional, connected, simply connected, and nilpotent subgroup, where $2\le k+1 \le N<\infty$, such that the quotient metric space $(Y/H,\overline{y})$ is isometric to $([0,\infty),0)$. Then the following hold.
    \begin{enumerate}
    
    \item\label{item:regular_set} $Y$ has rectifiable dimension $k+1$ and $(Y,H)$ is equivariantly homeomorphic to $([0,\infty)\times H,H)$, where $H$ acts by left-multiplication on the second factor. Under this identification, the $(k+1)$-regular set of $(Y,d,\meas)$ is
    $Y_+=(0,\infty)\times H$.

    \item\label{item:measure} There exists a $\CDe(0,N)$ density $\rho\colon [0,\infty)\to[0,\infty)$ such that $\meas=\rho(r)\di r\wedge\Omega$, where $\Omega$ is a left-invariant Haar measure on $H$.

    \item\label{item:continuous+sobolev} There exists a family $\{g_r\}_{r>0}$ of left-invariant Riemannian metrics on $H$ that is continuous and locally $W^{1,2}$ in $r$ such that $(Y,d)$ is the metric completion of $(Y_+,d_g)$, where $d_g$ is the distance induced by the Riemannian metric $g=\di r^2+g_r$.

    \item \label{item:distributional-Ricci}The space $(Y_+,g,\meas)$ has nonnegative distributional $N$-Bakry--Emery Ricci curvature.
    \end{enumerate}
\end{thm}

\begin{rem}
Theorem \ref{thm:RCD+nilpotent} may be seen as an extension of \cite{CMS26}*{Proposition 3.7} and \cite{NPZ26}*{Section 3}, where $\R$-actions and $S^1$-actions on $\RCD(0,N)$ spaces were considered, to the more general context of connected, simply connected, and nilpotent Lie groups.
\end{rem}

Throughout this section, $(Y,y,d,\meas)$ and $H$ will always denote an $\RCD(0,N)$ space and a $k$-dimensional connected, simply connected, and nilpotent Lie group satisfying the conditions of Theorem \ref{thm:RCD+nilpotent}. 

Below we highlight the key steps of our proof and compare them with \cite{NPZ26}*{Sections 3 and 4}, which followed a similar strategy in the simpler cases where $H=\R$ and $H=S^1$.

\noindent\textbf{Step 1: Rectifiable dimension and regular set of $Y$.} In Section \ref{sec:rectifiable_dimension}, we prove item \eqref{item:regular_set}. Similarly to \cite{NPZ26}, we study equivariant tangent cones to $(Y,H)$ in order to determine the rectifiable dimension and regular set of $(Y,d,\meas)$. In comparison to the case where $H=S^1$, the main difficulty comes from the higher dimension of the orbits.

\noindent\textbf{Step 2: Continuous Riemannian metric on $Y_+$.} In Section \ref{sec:continuous_riemannian_metric}, we prove the ``continuous'' part of item \eqref{item:continuous+sobolev}. First, similarly to \cite{NPZ26}, we use the H\"older continuity of tangent cones \cites{CoNa12,Deng20} and the Abresch--Gromoll excess estimate \cites{AG90,GM14} to conclude that one-parameter subgroups of $H$ give rise to rectifiable curves in each orbit in the regular set. As a result, we obtain homogeneous length metrics $d_r$ on $H$. Thanks to a result of Berestovskii \cite{Berestovskii88}, such length metrics are necessarily sub-Finsler. Since one-parameter subgroups are rectifiable, our length metrics are Finsler. By passing to the tangent cone, we conclude that they are Riemannian, i.e., $d_r=d_{g_r}$, where $g_r$ is a left-invariant Riemannian metric on $H$. The remainder of the section follows from an analysis of the infinitesimal structure of $Y_+$, similar to that in \cite{NPZ26}.

\noindent\textbf{Step 3: Density function of $\meas$.} In Section \ref{sec:measure_CD(0,N)}, we establish item \eqref{item:measure}. Our approach differs from \cite{NPZ26}, where we studied the cases $H=\R$ and $H=S^1$. The case $H=\R$ can be reduced to the $S^1$ case since $\R$ admits $\mathbb{Z}$ as a normal lattice, so that $\R/\mathbb{Z}=S^1$ acts on the quotient space $Y/\mathbb{Z}$. Since $S^1$ is compact, we were able to conclude using the results of \cite{Galaz-Garcia_18}. In the current setting, $H$ may not admit any lattice and, if it does, this lattice will generally not be normal. Therefore, we rely on the needle decomposition method established by Cavalletti and Mondino in \cites{CM17,CM_newformula}. First we study the disintegration of our measure w.r.t\@ the $1$-Lipschitz map $d_{\partial}(r,h)\coloneqq r$, where $(r,h)\in [0,\infty)\times H=Y$. We conclude using the $H$-invariance of needles, which are equipped with a $\CDe(0,N)$ density.

\noindent\textbf{Step 4: Improving to a $W^{1,2}_{loc}$ metric on $Y_+$ and $N$-Bakry--Emery Ricci curvature.} In Section \ref{sec:loc_sobolev_metric}, we use the previous two sections to prove the ``Sobolev'' part of item \eqref{item:continuous+sobolev}. The approach is similar to \cite{NPZ26} in spirit and the main complications arise from the non-commutativity of our group $H$. In our context, the exponential map $\exp\colon\mathfrak{h}\to H$ is a diffeomorphism, which gives rise to coordinate functions $x^1,\ldots,x^k$. Using the second-order calculus tools developed for $\RCD$ spaces (see \cite{gigli_nonsmooth} for an overview), we show that $g(\nabla_g x^i,\nabla_g x^j)\in W^{1,2}_{loc}(Y_+)$. Finally, based on the results of Mondino and Ryborz \cite{MondinoRyborz}, we obtain item \eqref{item:distributional-Ricci}.

\subsection{Rectifiable dimension and regular set of $Y$}\label{sec:rectifiable_dimension}

The goal of this subsection is proving item \eqref{item:regular_set} of Theorem \ref{thm:RCD+nilpotent}. 

Let $\pi:Y\to Y/H=[0,\infty)$ be the quotient map. We denote $Y_+\subseteq Y$ the pre-image $\pi^{-1}(0,\infty)$. We first prove that $Y$ is homeomorphic to $[0,\infty)\times H$ with $Y_+$ identified with $(0,\infty)\times H$. This part is straightforward.

\begin{lem}\label{lem:Y_homeo_RxH}
The equivariant space $(Y,H)$ is equivariantly homeomorphic to $([0,\infty)\times H,H)$, where $H$ acts by left-multiplication on the second factor, and $Y_+$ is identified with $(0,\infty)\times H$.
\end{lem}

\begin{proof}
   We first note that $H$ acts freely on $Y$; otherwise, $H$ would have a nontrivial compact isotropy subgroup, a contradiction to the assumption that $H$ is a connected and simply connected nilpotent Lie group. Next, we lift the ray in the quotient space $Y/H=[0,\infty)$ to $Y$ starting at $y$. We denote this horizontal ray by $\gamma$ and consider a map 
$$\varphi:[0,\infty)\times H\to Y,\quad (r,h)\mapsto h\cdot \gamma(r).$$
It is clear that $\varphi$ is an equivariant homeomorphism between $(Y,H)$ and $([0,\infty)\times H,H)$ and that $Y_+$ is identified with $(0,\infty)\times H$.
\end{proof}

Next, we study the rectifiable dimension and regular set of $Y$. For convenience, we separate them into two propositions below. 

\begin{prop}\label{prop:rect_dim_k+1}
  Under the conditions of Theorem \ref{thm:RCD+nilpotent}, the space $(Y,d,\meas)$ has rectifiable dimension $k+1$.
\end{prop}

\begin{prop}\label{prop:reg_set_Y+}
   The $(k+1)$-regular set of $Y$ is exactly $Y_+$. As a consequence, $Y_+$ is convex in $Y$.
\end{prop}

It is clear that item \eqref{item:regular_set} of Theorem \ref{thm:RCD+nilpotent} follows from Lemma \ref{lem:Y_homeo_RxH}, Proposition \ref{prop:rect_dim_k+1}, and Proposition \ref{prop:reg_set_Y+}. 

\begin{lem}\label{lem:Y+_regular}
   Let $m$ be the rectifiable dimension of $Y$. Then every point in $Y_+$ is $m$-regular.
\end{lem} 

\begin{proof}
  Suppose that $Y_+$ has a point $q$ that is not $m$-regular. Then by $H$-action, every point in the orbit $Hq$ is not $m$-regular. The orbit $Hq$ separates $Y$ into two disconnected open subsets, thus both have positive $\meas$-measure. On the other hand, the $m$-regular set has full $\meas$-measure \cite{BrueSemola20} and is path connected \cite{Deng20}. Hence we end in a desired contradiction.
\end{proof}

\begin{lem}\label{lem:blow_up_R}
   Let $q\in Y_+$ be an $m$-regular point. Then\\
   (1) the equivariant tangent cone of $(Y,H)$ at $q$ is unique and isomorphic to $(\R^{m},0,\R^{m-1})$, where $\R^{m-1}$ acts on $\R^m$ by translations.\\
   (2) for any $\R$-subgroup $L$ of $H$, the equivariant tangent cone of $(Y,L)$ at $q$ is unique and isomorphic to $(\R^{m},0,\R)$, where the limit $\R$-subgroup acts on $\R^m$ by translations. 
\end{lem}

\begin{proof}
   (1) The point $q\in Y_+$ projects to a $1$-regular point in $Y/H=[0,\infty)$. For any $r_i\to 0$, after passing to a subsequence, we have convergence
   \begin{equation*}
   \begin{CD}
    (r_i^{-1} Y,q,H) @>GH>> (\R^m,0,H_\infty) \\
	@VV\pi V @VV \pi V\\
	(r_i^{-1} [0,\infty),\pi(q)) @>GH>> (\R,0).
    \end{CD}
\end{equation*}
   By construction, the limit group $H_\infty$ is a closed and nilpotent subgroup of $\mathrm{Isom}(\R^m)$; moreover, $H_\infty$ acts transitively on an $\R^{m-1}$-factor of $\R^m$. Thanks to \cite{Pan21}*{Lemma 4.1}, $H_\infty=\R^{m-1}$ acts by translations. 

    (2) The proof is a mild modification of \cite{NPZ26}*{Proposition 3.5}. We give a sketch below for readers' convenience.

    We consider any equivariant tangent cone of $(Y,L\leq H)$ at $q$. Using (1) above, we have
    $$(r_i^{-1} Y,q,L\leq H) \overset{GH}\longrightarrow(\R^m,0,L_\infty\leq \R^{m-1}),$$
    where $r_i\to 0$ and $\R^{m-1}$ acts on $\R^m$ by translations. We write the group elements in $L\simeq \R$ by $v\in \R$. For each $i$, we set
    $$v_i=\min\{ v>0 \mid d(v\cdot q,q)=r_i \}.$$
    It is clear that $v_i\to 0$ as $i\to\infty$. Then we define a symmetric subset $S_i$ of $L$ by
    $$S_i=\{ v\in L \mid v\in[-v_i,v_i] \}$$
    and consider the convergence
    \begin{equation}\label{eq:eGH_rect_dim}
    (r_i^{-1}Y,q,v_i\in S_i\subseteq L\leq H)\overset{GH}\longrightarrow (\R^m,0,g\in A\subseteq L_\infty\leq \R^{m-1}),
    \end{equation}
    where $A$ is a closed symmetric subset of $L_\infty$. From construction it follows that $A\cdot 0\subseteq \overline{B_1}(0)$ and $d(g\cdot 0,0)=1$. We denote $z=g\cdot 0$.

    Following the same argument in Claim 1 of \cite{NPZ26}*{Proposition 3.5}, we can show that $A\cdot 0$ contains the set $\{ tz\in \R^m \mid t\in [-1,1] \}$, where $tz$ is written using the linear structure of $\R^m$. We continue to call this statement Claim 1 for convenience.

    By Claim 1 and (\ref{eq:eGH_rect_dim}), we have $L_\infty \cdot 0 \supseteq \R z$, where $\R z$ is written using the linear structure. We shall show that $L_\infty \cdot 0 = \R z$. We argue by contradiction and suppose that $L_\infty$ contains an element $h$ such that $h\cdot 0 \not\in \R z$. Since $L_\infty\leq \R^{m-1}$, every group element of $L_\infty$ acts by translation in $\R^m$. After replacing $h$ by a power of $h$, we can assume that $d(h\cdot 0,\R z)\ge 2$. Let $u_i\in L$ such that $u_i\overset{GH}\to h$ associated with (\ref{eq:eGH_rect_dim}). Then $u_i\to 0$ in $L\simeq \R$ by construction. We may also assume that $u_i>0$ after replacing by $h$ by $-h$, if necessary. 

    Using the same argument in Claim 2 of \cite{NPZ26}*{Proposition 3.5}, we obtain $u_i \gg v_i$. For each $i$, we define
    $$d_i=\max\{ d(v\cdot q,q) \mid v\in [v_i,u_i] \}.$$
    It is clear that $d_i\to 0$ because $u_i\to 0$ in $L\simeq \R$. Then $d_i\gg r_i$; see proof of Claim 3 in \cite{NPZ26}*{Proposition 3.5}. Next, we set
    $$T_i=\{ v\in L \mid v\in[-u_i,u_i] \}$$
    and consider the blow-up sequence by $d_i^{-1}\to \infty$:
    \begin{equation}\label{eq:eGH_rect_dim_rescal}
    (d_i^{-1}Y,q,u_i\in T_i \subseteq L\leq H)\overset{GH}\longrightarrow (\R^m,0,h'\in B'\subseteq L'_\infty\leq H'_\infty).
    \end{equation}
    By (1), $H'_\infty\simeq \R^{m-1}$ acts by translations on $\R^m$. Since $d_i\gg d(u_i\cdot q,q)$, we have $h'\cdot 0=0$ and thus $h'=\mathrm{id}$.

    Lastly, using the same argument in Claim 4 of \cite{NPZ26}*{Proposition 3.5}, one can show that $B'$ is a subgroup. By the choice of $d_i$, from (\ref{eq:eGH_rect_dim_rescal}) we see that $B'\cdot 0$ is contained in $\overline{B_1}(0)$ with some point in $B'\cdot 0$ lying on $\partial B_1(0)$. This is impossible for a subgroup $B'\leq H'_\infty\simeq \R^{m-1}$ that acts by translations. We end in a desired contradiction; in other words, we have $L_\infty\cdot 0 = \R z$. Since any element in $L_\infty$ acts by translation, we conclude that $L_\infty\simeq \R$ and the result follows.
\end{proof}

We are ready to prove Proposition \ref{prop:rect_dim_k+1}.

\begin{proof}[Proof of Proposition \ref{prop:rect_dim_k+1}]
   The simply connected nilpotent Lie group $H$ of dimension $k$ has a series of normal subgroups
   $$\{\mathrm{id}\}=H^0 \triangleleft H^1 \triangleleft ... \triangleleft H^k = H$$
   such that $H^{j+1}/H^j\simeq \R$. Let $r_i\to 0$ and we consider an equivariant tangent cone of $Y$ at an $m$-regular point $q\in Y_+$. For any sequence $r_i\to 0$, after passing to a subsequence, by Lemma \ref{lem:blow_up_R} we have
   $$(r^{-1}_i Y,q,H^1 \triangleleft ... \triangleleft H^k)\overset{GH}\longrightarrow (\R^m,0,\R=H^1_\infty\triangleleft ... \triangleleft H^k_\infty=\R^{m-1}).$$
   For each $j=1,...,k-1$, we consider the quotient space $Y/H^j$ with $H/H^j$-action. Because $H/H^j$ is a connected, simply connected, and nilpotent Lie group, and the quotient metric space is isometric to a ray
   $$(Y/H^j)/(H/H^j) = Y/H =[0,\infty),$$
   by Lemma \ref{lem:Y+_regular}, the point $q\in Y_+$ projects to $Y/H^j$ a regular point in $(Y/H^j)_+$. Applying Proposition \ref{prop:eGH_quotient} and Lemma \ref{lem:blow_up_R}(2) to $(Y/H^j,\bar{q},H/H^j)$ with $\R$-subgroup $H^{j+1}/H^j$, we obtain
   $$\left(r^{-1}_i (Y/H^j),\bar{q},(H^{j+1}/H^j)\right)\overset{GH}\longrightarrow \left(\R^{m_j},0,\R=(H^{j+1}_\infty/H^j_\infty)\right),$$
   where $m_j$ denotes the rectifiable dimension of $Y/H^j$. Therefore, the limit group $H_\infty^k=\R^{m-1}$ has a series of normal subgroups 
   $$\{\mathrm{id}\}=H^0_\infty \triangleleft H^1_\infty\triangleleft \dots \triangleleft H^k_\infty=\R^{m-1}$$
   such that each $H^{j+1}_\infty/H^j_\infty$ is isomorphic to $\R$. It follows that $H^k_\infty=\R^{m-1}$ has dimension $k$ and thus $Y$ has rectifiable dimension $m=k+1$. This completes the proof.
\end{proof}



Next, we work towards Proposition \ref{prop:reg_set_Y+} on the regular set of $Y$.

\begin{lem}\label{lem:nil_reflect}
   Let $G$ be a closed and nilpotent subgroup of $\mathrm{Isom}(\R^{k+1})$. Suppose that\\
   (1) the quotient metric space $(\R^{k+1}/G,\bar{0})$ is isometric to a ray $([0,\infty),0)$, and\\
   (2) $G$ admits a series of closed normal subgroups
   $$\{\mathrm{id}\}=G^0 \triangleleft G^1 \triangleleft ... \triangleleft G^k = G$$
   such that each $G^{j+1}/G^j$ contains a closed $\R$-subgroup.\\
   Then $G=\R^k\times \Z_2$, where $\R^k$-subgroup acts by translation and $\Z_2$-subgroup acts by reflection about the hyperplane $\R^k\cdot 0$.
\end{lem}

\begin{proof}
   We write elements in $\mathrm{Isom}(\R^{k+1})$ as $(A,v)\in O(k+1) \ltimes \R^{k+1}$. Let
   $$p:\mathrm{Isom}(\R^{k+1})\to O(k+1), \quad (A,v)\mapsto A$$
   be the projection map. Let $G_0$ be the identity component subgroup of $G$ and let $K$ be the closure of $p(G_0)$ in $O(k+1)$. It is clear that $K$ is a compact, connected, and nilpotent. Hence $K$ is a torus; in particular, it is abelian. Thanks to \cite{Pan25}*{Lemma 2.4}, this shows that $G_0$ is abelian. By assumption (2), we can write $G_0=\R^m\times \mathbb{T}$, where $m\ge k$ and $\mathbb{T}$ is a torus of dimension $\ge 0$. Let $L$ be any closed $\R^k$-subgroup of $G_0$. By linear algebra, it is clear that $L$ acts by translations in $\R^{k+1}$. It follows that $G_0=\R^k$ and $G=\R^k\times \Z_2$.
\end{proof}

\begin{proof}[Proof of Proposition \ref{prop:reg_set_Y+}]
   
   Using the isometric $H$-action, it suffices to show that the base point $y$ is not $(k+1)$-regular. We argue by contradiction and suppose otherwise. Then for any $r_i\to 0$, we have
   \begin{equation}\label{eq:y_singular}
   \begin{CD}
    (r_i^{-1} {Y},y,H) @>\mathrm{GH}>> (\R^{k+1},0,G) \\
	@VV\pi V @VV \pi V\\
	(r_i^{-1} [0,\infty),0) @>\mathrm{GH}>> ([0,\infty),0),
    \end{CD}
    \end{equation}
    where $G$ is a closed and nilpotent subgroup of $\mathrm{Isom}(\R^{k+1})$.

    We verify that $G$ satisfies the assumption (2) in Lemma \ref{lem:nil_reflect}. In fact, because $H$ is a connected, simply connected, and nilpotent group of dimension $k$, it admits a series of closed normal subgroups 
    $$\{\mathrm{id}\}=H^0 \triangleleft H^1\triangleleft \dots \triangleleft H^k=H$$
    such that each $H^{j+1}/H^j$ is isomorphic to $\R$. After passing to a subsequence, we consider convergence associated with (\ref{eq:y_singular}):
    $$(r_i^{-1} {Y},y,H^1\triangleleft \dots \triangleleft H^k)\overset{GH}\longrightarrow (\R^{k+1},0,G^1\triangleleft \dots \triangleleft G^k=G).$$
    For each $j=0,...,k-1$, we apply Proposition \ref{prop:eGH_quotient} and obtain the convergence of quotient actions
    $$\left(r_i^{-1} (Y/H^j),\bar{y}, (H^{j+1}/H^j)\simeq \R \right) \overset{GH}\longrightarrow (\R^{k+1}/G^j,\bar{0},G^{j+1}/G^j).$$
    By construction, $G^{j+1}/G^j$ contains a closed $\R$-subgroup. This verifies the claim. Alternatively, one may prove the claim by using the arguments in \cite{PanYe}*{Lemma 3.1}.

    Combining the claim with Lemma \ref{lem:nil_reflect}, we conclude that $G=\R^k\times \Z_2$. The rest of the proof, which rules out the possibility of this remaining case, is a clear modification of the proof of \cite{NPZ26}*{Proposition 3.8}. In fact, let $h\in G=\R^k\times \Z_2$ be the reflection element and let $h_i\in H$ such that $h_i\overset{GH}\to h$ associated with (\ref{eq:y_singular}). Because $H$ is a connected and simply connected nilpotent Lie group, we can choose $u_i\in H$ such that $u_i^2=h_i$. We set $d_i \coloneqq d(u_i\cdot y,y)$, which is positive because $H$-action is free. If the sequence $r_i^{-1}d_i$ is uniformly bounded, then $u_i\overset{GH}\to g\in G$ for some $g\in G$ associated with (\ref{eq:y_singular}). Then we have $g^2=h$ in $G=\R^k\times \Z_2$, which is impossible because $h$ is the reflection element. If $r_i^{-1}d_i\to \infty$ for some subsequence, then for this subsequence we consider the convergence
    $$(d_i^{-1}Y,y,H,u_i,h_i)\overset{GH}\longrightarrow (\R^{k+1},0,\R^k\times \Z_2,g',h').$$
    By construction, we have
    $$(g')^2=h',\quad d(g'\cdot 0,0)=1,\quad d(h'\cdot 0,0)=0,$$
    which is also impossible on $\R^{k+1}$ with $(\R^k\times \Z_2)$-action. 

    In conclusion, $y$ being $(k+1)$-regular is impossible. Together with Lemma \ref{lem:Y+_regular} and Proposition \ref{prop:rect_dim_k+1}, we conclude that the $(k+1)$-regular set of $Y$ is exactly $Y_+$. Lastly, the convexity of $Y_+$ follows from the openness of $Y_+$ and the continuity of tangent cones along the interior of any geodesic \cite{Deng20}.
\end{proof}

\subsection{Recovering a $C^0$ Riemannian metric on $Y_+$}\label{sec:continuous_riemannian_metric}

In this section, we prove the existence of a continuous Riemannian metric described in item \eqref{item:continuous+sobolev} of Theorem \ref{thm:RCD+nilpotent}. For convenience, we restate this goal as a proposition below.

\begin{prop}\label{prop:recover_cont_metric}
   Under the assumptions of Theorem \ref{thm:RCD+nilpotent}, the metric $d$ on $Y$ comes from the metric completion of 
   $$g=\di r^2 + g_r$$
   on the subset $Y_+=(0,\infty)\times H$, where $\{g_r\}_{r>0}$ is a family of left-invariant Riemannian metrics on $H$ which is continuous in $r$.
\end{prop}

As mentioned in Section \ref{sec:statements_nilpotent+RCD}, the strategy to prove Proposition \ref{prop:recover_cont_metric} follows a similar approach to that in \cite{NPZ26}*{Sections 3.3 and 3.4}, where we considered $S^1$-actions. 

For the first part of the proof (from Lemma \ref{lem:uniform_regular} to Lemma \ref{lem:proj_estimate} below), most arguments are almost identical to the ones in \cite{NPZ26} with minimal modifications. The goal of this part is showing that on any $H$-orbit in $Y_+$, any one-parameter orbit is rectifiable; moreover, the intrinsic metric on the $H$-orbit is locally $(1+\epsilon)$-biLipschitz to the extrinsic metric $d$. 

The subsequent part requires a different approach compared to \cite{NPZ26}, due to the higher dimension of $H$. Here, we apply a classical result of Berestovskii \cite{Berestovskii88}, which states that any homogeneous length metric on a connected Lie group is a left-invariant sub-Finsler metric, to obtain a family of sub-Finsler metrics on $H_r=\pi^{-1}(r)$, where $\pi:Y\to [0,\infty)$ is the quotient map and $r>0$. Then we use the tangent cone structure to further show that these sub-Finsler metrics are Riemannian. This gives a Riemannian metric $g=dr^2+g_r$ on $Y_+=(0,\infty)\times H$.

Lastly, we show that $g$ induces the same distance $d$ on $Y_+$. The proof in this part is similar to \cite{NPZ26}*{Section 3.4} with some mild modifications. 

\begin{lem}\label{lem:uniform_regular}
   Let $K$ be a compact subset of $Y_+$. Then for any $z_i\in K$ and any $d_i\to 0$ such that the sequence $(d_i^{-1}Y,z_i)$ is convergent, we have
   $$(d_i^{-1}Y,z_i)\overset{GH}\longrightarrow (\R^{k+1},0).$$
\end{lem} 

\begin{proof}
   The proof is almost identical to \cite{NPZ26}*{Lemma 3.10}, based on H\"older continuity of small balls along the horizontal ray \cite{Deng20} and Proposition \ref{prop:rect_dim_k+1}. The only minor change is the change from $\R^2$ in the proof of \cite{NPZ26}*{Lemma 3.10} to $\R^{k+1}$.
\end{proof}

Because $H$ is a connected and simply connected nilpotent Lie group, the Lie group exponential map $\exp: \mathfrak{h}\to H$ is a diffeomorphism. As a result, for any $h\not=\mathrm{id}$, the one-parameter subgroup through $h$ is uniquely determined. For convenience, we denote elements on the one-parameter subgroup by $th\coloneqq\exp(tV)$, where $t\in \R$ and $V=\exp^{-1}(h)$.

\begin{lem}\label{lem:deviation_estimate}
   Let $\epsilon>0$ and $y_0=(r_0,\mathrm{id})\in Y_+$, where $r_0>0$. Then there is a small $\delta>0$ such that for all $(r,h)\in B_\delta(y_0)$, it holds that
   $$d((r,h/2),\sigma)\le \epsilon\cdot d((r,\mathrm{id}),(r,h)),$$
   where $\sigma$ is a minimal geodesic from $(r,\mathrm{id})$ to $(r,h)$. 
\end{lem}

\begin{proof}
   The proof is a clear modification of \cite{NPZ26}*{Lemma 3.11}, using a standard contradiction argument and Lemma \ref{lem:uniform_regular}.
\end{proof}

\begin{lem}\label{lem:division_estimate}
  Let $\epsilon>0$ and $y_0=(r_0,\mathrm{id})\in Y_+$, where $r_0>0$. There is $\delta>0$ such that for all $(r,h)\in B_\delta(y_0)$ and all $m\in \mathbb{N}$, it holds that
  $$2^m \cdot d((r,\mathrm{id}),(r,h/2^m))\le (1+\epsilon) \cdot d((r,\mathrm{id}),(r,h)).$$
\end{lem}

\begin{proof}
   The proof follows verbatim from \cite{NPZ26}*{Lemma 3.13}, using Lemma \ref{lem:deviation_estimate} and the Abresch--Gromoll excess estimate \cites{AG90,GM14}.
\end{proof}

\begin{cor}\label{cor:orbit_length}
   Let $\epsilon>0$ and let $y_0=(r_0,\mathrm{id})\in Y_+$, where $r_0>0$. For $(r,h)\in Y_+$, we denote $\sigma_{r,h}$ the curve from $(r,\mathrm{id})$ to $(r,h)$ defined by $\sigma_{r,h}(t)=(r,th)$, where $t\in[0,1]$. Then there is $\delta>0$ such that
   $$\mathrm{length}_d(\sigma_{r,h})\le (1+\epsilon)\cdot d((r,\mathrm{id}),(r,h))$$
   for all $(r,h)\in B_\delta(y_0)$. 
\end{cor}

\begin{proof}
   This follows immediately from Lemma \ref{lem:division_estimate} and the isometric $H$-action; also see \cite{NPZ26}*{Corollary 3.15}.
\end{proof}

\begin{lem}\label{lem:pyth_estimate}
   Let $\epsilon>0$ and $y_0=(r_0,\mathrm{id})\in Y_+$, where $r_0>0$. Then there is $\delta>0$ such that for all $(r,h)\not=(r',h')\in B_\delta(y_0)$, it holds that
  $$1-\epsilon \le \dfrac{\left(d((r,h),(r,h'))^2+ |r-r'|^2 \right)^{1/2}}{d((r,h),(r',h'))}\le 1+\epsilon.$$
\end{lem}

\begin{proof}
   The proof mostly follows verbatim from \cite{NPZ26}*{Lemma 3.16}, using a contradiction argument. The only change is that the blow-up sequence involved in the proof now has limit
   $$(d_i^{-1}Y,(r_i,h_i),(r'_i,h'_i),H)\overset{GH}\longrightarrow (\R^{k+1},0,z,\R^k),$$
   where $\R^k$ acts by translations.
\end{proof}

For each $r>0$, we denote 
$$H_r=\{(r,h) \mid h\in H\}=H\cdot \gamma(r).$$
$H_r$ naturally carries an intrinsic metric from $d$; we denote this intrinsic metric on $H_r$ by $d_r$.

\begin{cor}\label{cor:dist_int_ext}
   Let $\epsilon>0$ and $y_0=(r_0,\mathrm{id})\in Y_+$, where $r_0>0$. Then there is $\delta>0$ such that 
   $$1 \le \dfrac{d_{r}(h,h')}{d((r,h),(r,h'))}\le 1+\epsilon$$
   for all $(r,h),(r,h')\in B_\delta(y_0)$.
\end{cor}

\begin{proof}
   This follows immediately from Corollary \ref{cor:orbit_length} and the isometric $H$-action on $H_r$. 
\end{proof}

In the lemma below, we show that for $r$ close to $r_0>0$, $(H_{r_0},d_{r_0})$ and $(H_r,d_r)$ are locally $(1+\epsilon)$-biLipschitz through the identity map.

\begin{lem}\label{lem:proj_estimate}
   Let $\epsilon>0$ and $y_0=(r_0,\mathrm{id})\in Y_+$, where $r_0>0$. Then there is $\delta>0$ such that 
   $$1-\epsilon \le \dfrac{d_r(h,h')}{d_{r_0}(h,h')}\le 1+\epsilon$$
   for all $(r,h)\not=(r,h')\in B_\delta(y_0)$.
\end{lem}

\begin{proof}
   By the isometric $H$-action, it suffices to prove the statement for $h'=\mathrm{id}$.

   \textbf{Claim:} There is $\delta>0$ such that
   \begin{equation}\label{eq:transl_bilip}
   1-\epsilon \le \dfrac{d_r(\mathrm{id},h)}{d_{r_0}(\mathrm{id},h)}\le 1+\epsilon
   \end{equation}
   holds for all $(r,h)\in B_\delta(y_0)$ with $d((r,\mathrm{id}),(r,h))\ge \delta/10$.

   Suppose that the claim fails, then there are a sequence $\delta_i\to 0$ and a sequence $(r_i,h_i)\in B_{\delta_i}(y_0)$ with $d((r_i,\mathrm{id}),(r_i,h_i))\ge \delta_i/10$ but (\ref{eq:transl_bilip}) does not hold. For convenience, we denote 
   $$z_i=(r_i,\mathrm{id}),\quad w_i=(r_i,h_i),\quad  d_i=d(z_i,w_i)\to 0.$$
   After passing to a subsequence, we have
   $$d_i^{-1}(r_i-r_0)\to s_0\in [-10,10].$$
   Let $w'_i=(r_0,h_i)$, then we apply Lemma \ref{lem:uniform_regular} and consider the convergence
   $$(d_i^{-1}Y, y_0,w'_i,z_i,w_i,H)\overset{GH}\longrightarrow (\R^{k+1},0,w',z,w,\R^k).$$
   It follows from the construction that $z=(s_0,0)$, $w=(s_0,v)\in \R\times \R^k$ for some $v\in \R^k$, then $w'=(0,v)$. Since $d(z,w)=d(0,w')$ in the limit space $\R^{k+1}$, we conclude that 
   $$\dfrac{d(z_i,w_i)}{d(y_0,w'_i)}\to 1.$$
   Together with Corollary \ref{cor:dist_int_ext}, the claim follows.

   Next, we prove that (\ref{eq:transl_bilip}) holds for all $(r,h)\in B_{\delta/2}(y_0)$ with $d((r,\mathrm{id}),(r,h))\le \delta/10$, where $\delta>0$ is a constant such that the claim above, Lemma \ref{lem:division_estimate}, and Corollary \ref{cor:dist_int_ext} hold. We choose an integer $m$ such that $d((r,\mathrm{id}),(r,(2^m)h)\in [\delta/4,\delta/2)$. We apply the claim to the pair $(r,\mathrm{id})$ and $(r,(2^m) h)$ and obtain
   \begin{equation*}
   1-\epsilon \le \dfrac{d_r(\mathrm{id},(2^m)h)}{d_{r_0}(\mathrm{id},(2^m)h)}\le 1+\epsilon.
   \end{equation*}
   Then thanks to Corollary \ref{cor:dist_int_ext} and Lemma \ref{lem:division_estimate}, we obtain (\ref{eq:transl_bilip}) and complete the proof.
\end{proof}

Next, we apply a result of Berestovskii \cite{Berestovskii88} to obtain a subFinsler metric from $(H_r,d_r)$. Then, we use Corollary \ref{cor:orbit_length} to show that this metric is Finsler, and use the tangent cone structure to show that it is indeed Riemannian.

\begin{lem}\label{lem:fiber_finsler}
   For each $r>0$, the intrinsic metric $d_r$ on  $H_r$ is a left-invariant Finsler metric.
\end{lem}

\begin{proof}
   By Corollary \ref{cor:orbit_length}, the space $(H_r,d_r)$ is a length metric space such that every pair of points has finite distance between them. It is clear that $(H_r,d_r)$ is homogeneous under $H$-action. We apply a result of Berestovskii \cite{Berestovskii88}*{Theorem 2} to deduce that $(H_r,d_r)$ comes from a left-invariant sub-Finsler metric on $H$. Furthermore, because for every $V\in \mathfrak{h}$, the curve
   $$\sigma:\R \to H_r,\quad t\mapsto \exp(tV)\cdot \gamma(r)$$
   is rectifiable by Corollary \ref{cor:orbit_length}, we conclude that the sub-Finsler metric on $H_r$ is indeed Finsler by \cite{Berestovskii88}*{Theorem 3}.
\end{proof}

\begin{prop}\label{prop:fiber_Riem}
   For each $r>0$, the intrinsic metric $d_r$ on $H_r$ is a left-invariant (and thus smooth) Riemannian metric.
\end{prop}

\begin{proof}
   Let $r_0>0$ and $y_0=(r_0,\mathrm{id})\in Y_+$. For each small $\delta>0$, we consider a map
   $$A_\delta: \left(B_{\sqrt{\delta}}(y_0),\delta^{-1}d\right)\to \left(\R\times H_{r_0},\delta^{-1}(d_E \times d_{r_0})\right),\quad (r,h)\mapsto (r,h),$$
   where $d_E\times d_{r_0}$ denotes the metric product of the Euclidean metric on $\R$ and $d_{r_0}$ on $H$. Combining Corollary \ref{cor:dist_int_ext}, Lemmas \ref{lem:pyth_estimate} and \ref{lem:proj_estimate}, for any $(r,h)\not=(r',h')\in B_{\sqrt{\delta}}(y_0)$, we have
   $$1-\epsilon(\delta) \le \dfrac{\left(d_{r_0}(h,h')^2+ |r-r'|^2 \right)^{1/2}}{d((r,h),(r',h'))}\le 1+\epsilon(\delta)$$
   where $\epsilon(\delta)\to 0$ as $\delta\to 0$. Hence $A_\delta$ is an $\epsilon(\delta)$-Gromov-Hausdorff approximation between $(\delta^{-1}B_{\sqrt{\delta}}(y_0),y_0)$ and $(\delta^{-1}(\R\times H_r),(r_0,\mathrm{id}))$. Let $\delta\to 0$, then by Lemma \ref{lem:fiber_finsler} we conclude that the tangent cone of $(Y,d)$ at $y_0$ is isometric to a metric product $\R \times C$, where $C$ is the (unique) tangent cone at $\mathrm{id}$ of the Finsler manifold $(H_{r_0},d_{r_0})$. Using the tangent cone structure at the $(k+1)$-regular point $y_0$ and the tangent cone structure on a Finsler manifold, we see that $\R^{k+1}$ is isometric to the metric product of $\R$ with $\R^k$ equipped with a norm. Thus the norm on the tangent space of $H_{r_0}$ at $\mathrm{id}$ must be an inner product. This verifies that the Finsler metric on $H_{r_0}$ is indeed Riemannian. Lastly, it follows from the left-invariance of $d_{r_0}$ that the Riemannian metric on $H_{r_0}$ is left-invariant (and thus smooth).
\end{proof}

\begin{notn}
With Proposition \ref{prop:fiber_Riem}, we denote $g_r$ the left-invariant Riemannian metric on $H_r$ that induces $d_r$.
\end{notn}

\begin{lem}\label{lem:g_r_cont}
   The family of Riemannian metrics $\{g_r\}_{r>0}$ is continuous in $r$.
\end{lem}

\begin{proof}
   Let $r_0>0$ and $V\in \mathfrak{h}-\{0\}$. It suffices to show that the function $f(r)\coloneqq g_r(V,V)^{1/2}$ is continuous at $r_0$. We continue to write $y_0=(r_0,\mathrm{id})\in Y_+$.
   
   Let $\epsilon>0$. We choose $\delta>0$ such that Corollary \ref{cor:orbit_length} holds. Then we fix a small $t_0>0$ such that the curve
   $$\sigma_r:[0,t_0] \to H_r,\quad \sigma_r(t)=\exp(tV)\cdot (r,\mathrm{id})=(r,\exp(tV))$$
   is contained in $B_\delta(y_0)$ for all $r\in (r_0-\delta/2,r_0+\delta/2)$. It is clear that $\mathrm{length}(\sigma_r)=f(r)\cdot t_0$. By the continuity of the function $r\mapsto d((r,\mathrm{id}),(r,\exp(t_0 V)))$, there is $\delta'\in (0,\delta/2)$ such that
   $$1-\epsilon\le\dfrac{d((r,\mathrm{id}),(r,\exp(t_0 V)))}{d((r_0,\mathrm{id}),(r_0,\exp(t_0 V)))}\le 1+\epsilon$$
   for all $r\in(r_0-\delta',r_0+\delta')$. Then we use Corollary \ref{cor:orbit_length} and conclude that
   \begin{align*}
      1+\epsilon \ge \dfrac{d((r,\mathrm{id}),\sigma_r(t_0))}{d((r_0,\mathrm{id}),\sigma_{r_0}(t_0))}\ge \dfrac{1}{1+\epsilon} \dfrac{\mathrm{length}(\sigma_r)}{\mathrm{length}(\sigma_{r_0})}=\dfrac{1}{1+\epsilon} \dfrac{f(r)}{f(r_0)}.
   \end{align*}
   Similarly, 
   \begin{align*}
      1-\epsilon \le \dfrac{d((r,\mathrm{id}),\sigma_r(t_0))}{d((r_0,\mathrm{id}),\sigma_{r_0}(t_0))}\le (1+\epsilon) \dfrac{\mathrm{length}(\sigma_r)}{\mathrm{length}(\sigma_{r_0})}=(1+\epsilon) \dfrac{f(r)}{f(r_0)}.
   \end{align*}
   Hence 
   $$\dfrac{1-\epsilon}{1+\epsilon}\le \dfrac{f(r)}{f(r_0)}\le (1+\epsilon)^2$$
   for all $r\in (r_0-\delta',r_0+\delta')$ and the result follows.
\end{proof}

\begin{cor}\label{cor:local_bilip_prod}
   Let $\epsilon>0$ and $y_0=(r_0,\mathrm{id})\in Y_+$, where $r_0>0$. Then there is $\delta>0$ such that $B_\delta(y_0)$ is $(1+\epsilon)$-biLipschitz to a $\delta$-ball in $(r_0-\delta,r_0+\delta)\times H$ centered at $(r_0,\mathrm{id})$ with product metric $dr^2 + g_{r_0}$.
\end{cor}

\begin{proof}
   This is an immediate consequence of Lemma \ref{lem:g_r_cont}.
\end{proof}

\begin{lem}\label{lem:norm=limit}
   Let $y_0=(r_0,\mathrm{id})\in Y_+$, where $r_0>0$, and $W\in T_{y_0}Y_+-\{0\}$. We write $W=(a,V)\in \R\times T_{\mathrm{id}}H$. Then
   $$\|W\|_g = \lim\limits_{t\to 0}\dfrac{d(y_0,c_W(t))}{|t|},$$
   where $g=dr^2+g_r$ and $c_W(t)=(r_0+at,\exp(tV))$.
\end{lem}

\begin{proof}
   When $|t|\not=0$ is sufficiently small, by Lemma \ref{lem:pyth_estimate}, we have
   \begin{equation}\label{eq:lem_norm_limit_1}
   \left|\dfrac{t^2a^2+d((r_0,\mathrm{id}),(r_0,\exp(tV)))^2}{d(y_0,c_W(t))^2}-1\right|\le \epsilon(t),
   \end{equation}
   where $\epsilon(t)\to 0^+$ as $t\to 0.$ We denote $\sigma(t)=(r_0,\exp(tV))$ the one-parameter orbit at $\gamma(r_0)$ in the direction of $V$. Then $$\mathrm{length}(\sigma|_{[0,t]})=|t|\cdot g_{r_0}(V,V)^{1/2},\quad 1\le \dfrac{\mathrm{length}(\sigma_{[0,t]})}{d((r_0,\mathrm{id}),(r_0,\exp(tV)))}\le 1+\epsilon(t),$$
   where the second inequality is a consequence of Lemma \ref{cor:orbit_length}. Together with (\ref{eq:lem_norm_limit_1}), they yield the desired equality.
\end{proof}

\begin{lem}\label{lem:metric_deriv=usual_deriv}
  Let $\gamma:[0,1]\to (Y_+,d)$ be a Lipschitz curve. Let $t\in [0,1]$ such that both the tangent vector $\gamma'(t)$ and the metric derivative $|\dot\gamma(t)|_d$ exist. Then
  $$\| \gamma'(t) \|_g=|\dot\gamma(t)|_d.$$
\end{lem}

\begin{proof}
   The proof is a mild modification of \cite{NPZ26}*{Lemma 3.19}. Given the differences in the choice of curves, we give a proof here.
   
   Let $\epsilon>0$. For convenience, given the fixed $t\in [0,1]$, we denote $$\gamma(t)=y_0=(r_0,\mathrm{id}),\quad \gamma'(t)=W=(a,V)\in T_{y_0} Y_+.$$
   We consider $\bar{g}=dr^2+g_{r_0}$ a metric product on $(0,\infty)\times H$, and
   $\sigma_W(s)=(r_0+sa,\exp(sV))$ the curve considered in Lemma 
   \ref{lem:norm=limit}.
   
   Because the curve $\sigma_W$ has tangent vector $W$ at $t$, there is $\delta>0$ such that
   $$d_{\bar{g}}(\gamma(t+s),\sigma_W(s))\le \epsilon |s|$$
   for all $s\in (-\delta,\delta)$. Thanks to Corollary \ref{cor:local_bilip_prod}, we can further shrink $\delta>0$ and obtain
   $$d(\gamma(t+s),\sigma_W(s))\le 2\epsilon |s|.$$
   Then 
   $$\dfrac{d(\gamma(t+s),y_0)}{|s|}\le \dfrac{d(\gamma(t+s),\sigma_W(s))}{|s|}+\dfrac{d(\sigma_W(s),y_0)}{|s|}\le 2\epsilon + \dfrac{d(\sigma_W(s),y_0)}{|s|}.$$
   Letting $s\to 0$ and then $\epsilon\to 0$, we get $|\dot{\gamma}(t)|_d\le \|W\|_g$ by using Lemma \ref{lem:norm=limit}.
    
   For the other direction of the inequality, we estimate
   $$\dfrac{d(y_0,\sigma_W(s))}{|s|}\le\dfrac{d(y_0,\gamma(t+s))}{|s|}+\dfrac{d(\gamma(t+s),\sigma_W(s))}{|s|}\le \dfrac{d(y_0,\gamma(t+s))}{|s|} + 2\epsilon.$$
   Then $\|W\|_g \le |\dot{\gamma}(t)|_d$ follows from Lemma \ref{lem:norm=limit}.
\end{proof}

\begin{prop}\label{prop:d=g_on_Y+}
    $d=d_g$ on $Y_+$.
\end{prop}

\begin{proof}
   The proof is almost identical to that of \cite{NPZ26}*{Proposition 3.21}, using Lemma \ref{lem:metric_deriv=usual_deriv} and the convexity of $Y_+$ (Proposition \ref{prop:reg_set_Y+}).
\end{proof}

Now Proposition \ref{prop:recover_cont_metric} follows immediately from Proposition \ref{prop:d=g_on_Y+} because the subset $Y_+$ is open and dense in $Y$.

\subsection{Recovering a $\CDe(0,N)$ density function on $Y$}\label{sec:measure_CD(0,N)}

We recall that $(Y,y,d,\meas)$ and $H$ respectively denote an $\RCD(0,N)$ space and a $k$-dimensional connected, simply connected, and nilpotent Lie group satisfying the conditions of Theorem \ref{thm:RCD+nilpotent}. Thanks to Lemma \ref{lem:Y_homeo_RxH}, we may assume without loss of generality that
\[
Y=[0,\infty)\times H\quad\text{and}\quad Y_+=(0,\infty)\times H,
\]
where $Y_+$ is the pre-image of $(0,\infty)$ under the quotient map $Y\to Y/H=[0,\infty)$.

From now on, we fix a left-invariant frame $\{X_1,\ldots, X_k\}$ on $H$. We define a left-invariant Haar measure $\Omega$ on $H$ by
\[
\Omega\coloneqq\theta^1\wedge\ldots\wedge \theta^k,
\]
where $\{\theta^1,\ldots,\theta^k\}$ is the dual co-frame.

\begin{rem}
    Observe that since $H$ is connected and nilpotent, it is unimodular, i.e., any left-invariant Haar measure is also right-invariant (see \cite{folland_course_2016}*{Proposition 2.30}).
\end{rem}

The goal of this section is to establish the following description of the measure $\meas$, which corresponds to item \eqref{item:measure} of Theorem \ref{thm:RCD+nilpotent}.

\begin{prop}\label{prop:measCD(0,N)}
We have:
\[
\meas=\rho(r)\di r\wedge \Omega,
\]
where $\rho$ is a nonnegative $\CDe(0,N)$ density on $[0,\infty)$, i.e., $\rho^{\frac{1}{N-1}}$ is concave.
\end{prop}

\begin{rem}
    {A $\CDe(0,N)$ density on $[0,\infty)$ is necessarily continuous on $[0,\infty)$ and locally Lipschitz on $(0,\infty)$.}
\end{rem}

In \cite{NPZ26}, we studied the case $H\simeq \R$. In that case, $H$ admits a normal lattice $\mathbb{Z}\triangleleft H$; hence, $Y/\mathbb{Z}$ admits a compact Lie group action. The results of \cite{Galaz-Garcia_18} allowed us to determine the measure on $Y/\mathbb{Z}$ and conclude since $Y$ is the universal cover of $Y/\mathbb{Z}$.

Under the conditions of Theorem \ref{thm:RCD+nilpotent}, the situation is more complex since our group may not admit a lattice and, if it does, this lattice is generally not a normal subgroup. Therefore, we use a different strategy to prove Proposition \ref{prop:measCD(0,N)}. Our approach relies on the needle decomposition developed by Cavalletti--Mondino \cites{CM17,CM_newformula} in the context of $\CDe(0,N)$ spaces, we refer the reader to their papers for a detailed exposition.

The first step to proving Proposition \ref{prop:measCD(0,N)} is the following lemma, which establishes a disintegration formula for our measure.

\begin{lem}\label{lem:measure_disintegration}
    There exists a family $\{\varphi_h\}_{h\in H}$ of nonnegative measurable functions on $[0,\infty)$ and a probability measure $\mathfrak{q}$ on $H$ such that
    \[
    \meas = \int_{h\in H}\varphi_h(r)\di r\di\mathfrak{q}(h),
    \]
    where $\varphi_h$ is a $\CDe(0,N)$ density on $[0,\infty)$ for $\mathfrak{q}$-a.e.\@ $h\in H$.    
\end{lem}

\begin{proof}
    Our first step is to determine the needle decomposition associated with the $1$-Lipschitz function $d_{\partial}(x)\coloneqq d(x,\partial Y)$ ($x\in Y$). Since we identified $Y=[0,\infty)\times H$ and $\partial Y=\{0\}\times H$, we have:
    \[
    d_{\partial}(r,h)= r.
    \]
    In \cite{NPZ24}*{Section 4.3}, we studied the case of a free $\R$-action with ray quotient on an $\RCD(0,N)$ space. Our past approach carries over verbatim to the current setting, where $\R$ is replaced by the Lie group $H$, which also acts freely. In particular, the arguments of \cite{NPZ24}*{Lemma 4.31}, which rely on the non-branchedness of $\RCD$ spaces established in \cite{Deng20}, imply the following:
    \begin{itemize}
        \item the set of branching points $A_+\cup A_-$ associated with $d_{\partial}$ is empty,
        \item the non-branching transport set $\mathcal{T}_{d_{\partial}}^{nb}$ is the whole space $Y$,
        \item given $x=(r,h)\in Y$, the needle $R_{d_{\partial}}(x)$ through $x$ coincides with $[0,\infty)\times \{h\}$.
    \end{itemize}
    Consequently, the map
    \[
    s\colon (r,h)\in [0,\infty)\times H\mapsto (1,h)\in \{1\}\times H,
    \]
    is a continuous selection map, i.e., $x\in Y$ and $y\in Y$ lie on a common needle if and only if $s(x)=s(y)$. As a result, we may identify the quotient set $\mathcal{T}_{d_{\partial}}^{nb}/R_{d_{\partial}}$ with $\{1\}\times H$, which we then identify with $H$. The lemma follows thanks to Cavalletti--Mondino's results \cite{CM_newformula}*{Theorems 3.4 and 3.6}.
\end{proof}

In order to prove Proposition \ref{prop:measCD(0,N)}, we need two classical lemmas, whose proof we provide for completeness. The first lemma states that, if a locally integrable function on $H$ is $k$-periodic almost everywhere for any $k\in H$, then it is constant.

\begin{lem}\label{lem: translation invariant a.e. implies constant}
Let $f\in L^1_{loc}(H,\Omega)$ and assume that, for every $k\in H$, $f(k^{-1}h)=f(h)$ for $\Omega$-a.e.\ $h\in H$. Then there exists $t\in \R$ such that $f(h)=t$ for $\Omega$-a.e.\ $h\in H$.
\end{lem}

\begin{proof}
    We recall that, in our setting, \(\exp\colon\mathfrak h\to H\) is a diffeomorphism. Let us fix a nonnegative and not identically zero function $\rho\in\mathcal{C}_c^{\infty}(H)$. For every integer $n\ge1$, define
    \[
    \rho_n(\exp(V))\coloneqq c_nn^k\rho(\exp(nV)),\quad V\in\mathfrak{h},
    \]
    where $k=\dim(H)$ and $c_n$ is chosen so that $\int\rho_n\di\Omega=1$. As a result, the sequence $\{\rho_n\}$ is a smooth approximate identity. For every integer $n\ge 1$, we introduce:
    \[
    f\ast\rho_n\colon\ x\in H\mapsto \int_Hf(xy^{-1})\rho_n(y)\di\Omega(y)
    \]
    and refer the reader to \cite{folland_course_2016}*{Chapter 2} for an introduction to convolution in the context of Lie groups equipped with a left-invariant Haar measure. Observe that, for every $k\in H$, we have
    \begin{equation}\label{eq:translation_invariant_convolution}
       f\ast \rho_n(k^{-1}x)=\int_Hf(k^{-1}xy^{-1})\rho_n(y)\di\Omega(y)=\int_Hf(xy^{-1})\rho_n(y)\di\Omega(y)=f\ast \rho_n(x) 
    \end{equation}
    thanks to the assumption of our lemma. Moreover, the following equality holds
    \[
    f\ast\rho_n(x)=\int_Hf(y)\rho_n(y^{-1}x)\di\Omega(y)
    \]
    since $H$ is unimodular (see \cite{folland_course_2016}*{(2.36)}). As a result, since $\rho_n\in\mathcal{C}_c^{\infty}(H)$, it is easily seen that $f\ast \rho_n$ is locally Lipschitz; thus, $f\ast \rho_n\equiv t_n$ for some $t_n\in \R$, thanks to \eqref{eq:translation_invariant_convolution}.
    
    Then, using that $\mathrm{Spt}(\rho_n)\subset B_{C/n}(e)$ and $\lvert \rho_n\rvert_{\infty}\le C\Omega(B_{1/n}(e))^{-1}$, we easily see that, for $x\in H$ Lebesgue point of $f$, we have $t_n=f\ast\rho_n(x)\to f(x)$ as $n\to\infty$. However, since $f\in L^1_{loc}(H)$, then $\Omega$-a.e.\@ $x\in H$ is a Lebesgue point of $f$. Therefore, $\{t_n\}$ converge to some limit $t\in \R$ and $f(x)=t$ for $\Omega$-a.e.\@ $x\in H$.
\end{proof}

In the proof of Proposition \ref{prop:measCD(0,N)}, we will encounter the situation where, given any bounded Borel set $B\subset [0,\infty)$ and $k\in H$, the following equality holds for $\Omega$-a.e.\@ $h\in H$:
\[
\int_B\psi_h(r)\di r=\int_B\psi_{k^{-1}h}(r)\di r,
\]
where $\psi_h$ is a family of measurable functions on $[0,\infty)$. Building on the previous lemma, our next lemma asserts that we may assume
\[
\int_B\psi_{h_1}(r)\di r=\int_B\psi_{h_2}(r)\di r
\]
for every Borel set $B\subset[0,\infty)$ and where $h_1,h_2$ lie in a set of full measure $E\subset H$.

\begin{lem}\label{lem: technical measure theory lemma}
Let $\{\psi_h\}_{h\in H}$ be a family of measurable functions $\psi_h:[0,\infty)\to[0,\infty)$ such that, for every bounded Borel set $B\subset[0,\infty)$, the following assumptions hold:
\begin{enumerate}
\item the map $F_B\colon h\in H\mapsto \int_B \psi_h(r)\di r\in[0,\infty)$ is in $L^1_{loc}(H,\Omega)$,
\item for every $k\in H$, $F_B(k^{-1}h)=F_B(h)$ for $\Omega$-a.e.\@ $h\in H$.
\end{enumerate}
Then, there exists a set of full-measure $E\subset  H$ such that, for every $h_1,h_2\in E$, we have:
\[
\int_B \psi_{h_1}(r)\di r=\int_B \psi_{h_2}(r)\di r
\qquad\text{for every Borel set }B\subset[0,\infty).
\]
\end{lem}

\begin{proof}
We denote by
\[
\mathcal G \coloneqq \bigl\{[a,b)\subset[0,\infty): 0\le a<b,\ a,b\in\mathbb{Q}\bigr\}
\]
the countable family of half-open intervals with rational endpoints, which generates the Borel \(\sigma\)-algebra of \([0,\infty)\), i.e., $\sigma(\mathcal{G})=\mathcal{B}([0,\infty))$. For every $B\in\mathcal G$, we have $F_B\in L^1_{loc}(H,\Omega)$ and, for every $k\in H$, $F_B(k^{-1}h)=F_B(h)$ for $\Omega$-a.e.\@ $h\in H$. Therefore, thanks to Lemma \ref{lem: translation invariant a.e. implies constant}, there exists a constant $c(B)\in[0,\infty)$ and a set of full-measure $E_B\subset  H$ such that $F_B(h)=c(B)$ for every $h\in E_B$. We denote $E\coloneqq\cap_{B\in\mathcal{G}}E_B$ (which is still of full-measure) and note that for every $B\in\mathcal{G}$ and $h\in E$, $F_B(h)=c(B)$. For every $h\in E$, we consider the $\sigma$-finite measure on $[0,\infty)$ defined by:
$$
\mu_h(B)\coloneqq \int_B \psi_h(r)\di r=F_B(h),\qquad B\subset[0,\infty)\ \text{Borel}.
$$
Let $h_1,h_2\in E$ and note that, by our previous step, $\mu_{h_1}(B)=\mu_{h_2}(B)=c(B)$ for every $B\in\mathcal{G}$. As a result, for every $R\in \mathbb{Q}$ positive,
$$
\mathcal{D}_R\coloneqq\{B\subset[0,\infty)\ \text{Borel},\ {\mu_{h_1}}(B\cap[0,R))={\mu_{h_2}}(B\cap[0,R))\}
$$
is a $\lambda$-system containing the $\pi$-system $\mathcal{G}$. By the $\pi-\lambda$ theorem, $\mathcal{D}_R$ contains $\sigma(\mathcal{G})=\mathcal{B}([0,\infty))$. Therefore, for every Borel set $B\subset[0,\infty)$ and $R\in\mathbb{Q}$ positive, $\mu_{h_1}(B\cap[0,R))={\mu_{h_2}}(B\cap[0,R))$. Passing to the limit as $R\to\infty$ implies that, for every Borel $B\in [0,\infty)$,
\[
\int_B \psi_{h_1}(r)\di r=\int_B \psi_{h_2}(r)\di r,
\]
which concludes the proof.
\end{proof}

We may now prove Proposition \ref{prop:measCD(0,N)}.

\begin{proof}[Proof of Proposition \ref{prop:measCD(0,N)}]
    As a result of Lemma \ref{lem:measure_disintegration}, we have $\meas = \int_{h\in H}\meas_h\di\mathfrak{q}(h)$, where $\mathfrak{q}$ is a probability measure on $ H$ and $\meas_h = \varphi_h(r)\di r$, where $\varphi_h$ is a nonnegative $\mathrm{CD}(0,N)$ density on $[0,\infty)$ for $\mathfrak{q}$-a.e.\@ $h\in H$. First, we show that $\mathfrak{q}$ is equivalent to the Haar measure $\Omega$.
    
    Let us denote $\Sigma\coloneqq[0,1]\times H$ and define $\hat{\mathfrak{q}}\coloneqq {p_2}_{\#}\meas\llcorner{\Sigma}$, where $p_2(r,h)=h$. Observe that $\hat{\mathfrak{q}}$ is $\sigma$-finite since, given any compact subset $K\subset H$, we have $\hat{\mathfrak{q}}(K)=\meas([0,1]\times K)<\infty$. Moreover, $\Sigma$ and $\meas$ are both $H$-invariant. Thus, $\hat{\mathfrak{q}}$ is a left-invariant Haar measure on $H$. As a result, there exists $c>0$ such that $\di \hat{\mathfrak{q}}(h) = c \di \Omega(h)$. Finally, given any Borel set $S\subset H$, we have $\hat{\mathfrak{q}}(S) = \meas([0,1]\times S) = \int_{h \in S}\meas_h([0,1])\di\mathfrak{q}(h)$. In particular, we have $\di \hat{\mathfrak{q}} = \mu(h)\di\mathfrak{q}$, where $\mu(h)=\meas_h([0,1])$ is measurable and positive $\mathfrak{q}$-a.e.\@ since $\meas$ has full-support. Therefore, $\di\mathfrak{q}(h)=\frac{c}{\mu(h)}\di \Omega(h)$ and we obtain
    \begin{equation}\label{eq:Haar_disintegration}
            \meas=\int_{h\in H}\psi_h\di \Omega(h),
    \end{equation}
    where $\psi_h\coloneqq\frac{c}{\mu(h)}\varphi_h$ is a $\CDe(0,N)$ density on $[0,\infty)$ for $\Omega$-a.e.\@ $h\in H$.

    We now use the $H$-invariance of $\meas$. Given any $h,k\in H$ and Borel set $A\subset Y$, we denote $L_{k}(h) \coloneqq kh$ and $A_h\coloneqq\{r\in[0,\infty),\ (r,h)\in A\}$. Observe that, since $\meas$ is $H$-invariant, then, for every Borel set $A\subset Y$, we have $\meas(A)=\meas(k^{-1}A)$. Consequently, using the disintegration formula \eqref{eq:Haar_disintegration}, we have  \begin{equation}\label{eq: h_v-w vs h_v}
    \begin{aligned} \int_{h\in H} \int_{r\in A_h}\psi_h(r)\di r \di\Omega(h) &=\int_{h\in H} \int_{r\in(k^{-1}A)_h}\psi_h(r)\di r \di\Omega(h)
    \\
    &=\int_{h\in H} \int_{r\in A_{kh}}\psi_h(r)\di r \di\Omega(h)\\ &=\int_{h\in H} \int_{r\in A_h}\psi_{k^{-1}h}(r)\di r \di ({L_{k}}_{\#}\Omega)(h)\\
    &=\int_{h\in H} \int_{r\in A_h}\psi_{k^{-1}h}(r)\di r \di\Omega(h),
    \end{aligned}
    \end{equation}
    where we used the left-invariance of $\Omega$ for the last equality. Therefore, for every Borel sets $S\subset H$ and $B\subset[0,\infty)$, we have
    \[
    \meas(B\times S)=\int_{h\in S}F_B(h)\di\Omega(h)=\int_{h\in S}F_B(k^{-1}h) \di\Omega(h),
    \]
    where $F_B\colon H\to [0,\infty]$ is the measurable function defined by $F_B(h)\coloneqq \int_B\psi_h(r)\di r$, for $h\in H$. Thus, for every $k\in H$ and Borel set $B\subset[0,\infty)$, we have:
    \[
    F_B(h) = F_B(k^{-1}h),\quad \text{for }\Omega\text{-a.e. } h\in H.
    \]
    Since $\meas$ is finite on bounded sets, then, for every bounded Borel set $B\subset [0,\infty)$, $F_B$ lies in $L^1_{loc}(H,\Omega)$. Therefore, as a consequence of Lemma \ref{lem: technical measure theory lemma}, and since $\psi_h$ is a $\mathrm{CD}(0,N)$ density for $\Omega$-a.e.\@ $h\in H$, there exists a set of full measure $E\subset H$ and $h_0\in E$ such that, for every $h\in E$ and Borel set $B\subset[0,\infty)$, we have $F_B(h)=F_B(h_0)$, where $\psi_{h_0}$ is a $\mathrm{CD}(0,N)$ density.
    
    To conclude, observe that, given any Borel sets $S\subset H$ and $B\subset[0,\infty)$, we have:
    \begin{equation*}
    \meas(B\times S) = \int_{h\in S}F_B(h) \di\Omega(h) = \int_{h\in S}F_B(h_0) \di\Omega(h) = \int_{h\in S}\int_{r\in B} \psi_{h_0}(r)\di r\di\Omega(h),
\end{equation*}
i.e., $\meas = \rho(r)\di r\wedge\di\Omega$, where $\rho(r)\coloneqq \psi_{h_0}(r)$ is a $\mathrm{CD}(0,N)$ density.
\end{proof}

\subsection{$W^{1,2}_{loc}$ metric on $Y_+$ and $\mathrm{BE}(0,N)$ inequality}\label{sec:loc_sobolev_metric}

We recall that $(Y,y,d,\meas)$ and $H$ respectively denote an $\RCD(0,N)$ space and a $k$-dimensional connected, simply connected, and nilpotent Lie group satisfying the conditions of Theorem \ref{thm:RCD+nilpotent}. Thanks to Lemma \ref{lem:g_r_cont} and Proposition \ref{prop:d=g_on_Y+}, we have
\[
d_{\lvert Y_+}=d_g,\quad g=\di r^2+g_r,
\]
where $\{g_r\}_{r>0}$ is a family of left-invariant Riemannian metrics on $H$ which is continuous in $r$. According to Proposition \ref{prop:measCD(0,N)}, we have
\[
\meas\llcorner Y_+=\rho(r)\di r\wedge\Omega,\qquad \Omega=\theta^1\wedge\ldots\wedge\theta^k,
\]
where $\{X_1,\ldots,X_k\}$ is a left-invariant frame on $H$, $\{\theta^{1},\ldots,\theta^k\}$ is the dual co-frame, and $\rho$ is a $\CDe(0,N)$ density on $[0,\infty)$. Let us denote by
\[
G(r)=[G_{ij}(r)]\coloneqq[ g_r(X_i,X_j)]
\]
the matrix of $g_r$ in the frame $\{X_1,\ldots,X_k\}$ and $G(r)^{-1}=[G^{ij}(r)]$ the matrix of the dual metric in the dual co-frame $\{\theta^1,\ldots,\theta^k\}$.

\begin{rem}\label{rem:g(nabla u,nabla v)}
    Given $u\in\mathcal{C}^{\infty}(Y_+)$, we have $\di u=\partial_ru\di r+X_i(u)\theta^i$. Thus, we have 
    \[
    g(\nabla_g u,\nabla_gv)=\partial_r u\: \partial_rv+G^{ij}(r)X_i(u)X_j(v),
    \]
    for every $v\in\mathcal{C}^{\infty}(Y_+)$.
\end{rem}

The goal of this section is to establish items \eqref{item:continuous+sobolev} and \eqref{item:distributional-Ricci} of Theorem \ref{thm:RCD+nilpotent}. We first prove the proposition below, which corresponds to item \eqref{item:continuous+sobolev}. We will then obtain item \eqref{item:distributional-Ricci} at the very end of the section.

\begin{prop}\label{prop:W^12G(r)}
    For every $1\le i,j\le k$, we have $G_{ij}(r)=g_r(X_i,X_j)\in W^{1,2}_{loc}\left((0,\infty)\right)$.
\end{prop}

We first need the following lemma which relates local Lipschitz constants to gradients in the context of continuous Riemannian metrics.

\begin{lem}\label{lem: weak gradient vs lip}
    If $\varphi\in\Lip_{loc}(Y_+)$, then $\varphi$ is differentiable $\meas$-a.e.\ on $Y_+$. Moreover, if $\varphi$ is differentiable at $x\in Y_+$, then $\lip(\varphi)(x) = |\nabla_g\varphi|(x)$.
\end{lem}

\begin{proof}
    Let us identify $H$ with $\R^k$. We recall that $d=d_g$ on $Y_+$ by Proposition \ref{prop:d=g_on_Y+}. Thanks to \cite{burtscher_length_2015}*{Theorem 4.5}, $\varphi$ is locally Lipschitz with respect to the standard Euclidean distance on $Y_+$. By Rademacher's theorem, $\varphi$ is differentiable $\lambda_{k+1}$-a.e.\@ on $Y_+$, where $\lambda_{k+1}$ is the standard Lebesgue measure on $Y_+$. However, by Proposition \ref{prop:measCD(0,N)}, we have $\meas\sim\lambda_{k+1}$ on $Y_+$. Therefore, $\varphi$ is differentiable $\meas$-a.e. on $Y_+$.

    If $\varphi$ is differentiable at $(r_0,x_0)\in Y_+$, we want to show that $\lip(\varphi)(r_0,x_0)= |\nabla_g\varphi|(r_0,x_0)$. The idea is to compare $g$ with the product metric $\di r^2+g_{r_0}$ in a neighborhood of $(r_0,x_0)$. The arguments of \cite{NPZ26}*{Lemma 3.28} carry over verbatim replacing $S^1$ with $H$.
\end{proof}

\begin{rem}\label{rem:smooth_implies_Lip_loc}
    Observe that $\mathcal{C}^{\infty}(Y_+)\subset \mathrm{Lip}_{loc}(Y_+)$. Indeed, if $\varphi$ is smooth on $Y_+$, then it is locally Lipschitz with respect to to $d_{{E}}$, where $d_{{E}}$ is the standard Euclidean distance on $Y_+$. Therefore, thanks to \cite{Burtscher15}*{Theorem 4.5}, $\varphi$ is locally Lipschitz with respect to $d_g$.
\end{rem}

\begin{rem}\label{rem:RCD<>toRiem<>}
    As a result of Remark \ref{rem:Df=lip(f)}, for every $\varphi,\psi\in \mathrm{Lip}(Y)$, the following holds $\meas$-a.e.\@ on $Y_+$ 
    \[
        \langle \nabla \varphi,\nabla \psi\rangle\coloneqq\frac14\Big(|D(\varphi+\psi)|_w^2-|D(\varphi-\psi)|_w^2\Big)=\frac14\Big(\lip(\varphi+\psi)^2-\lip(\varphi-\psi)^2\Big).
    \]
    However, thanks to Lemma \ref{lem: weak gradient vs lip}, we have:
    \[
    \frac14\left(\lip(\varphi+\psi)^2-\lip(\varphi-\psi)^2\right)=\frac14\left(|\nabla_g(\varphi+\psi)|^2-|\nabla_g(\varphi-\psi)|^2\right)=g(\nabla_g\varphi,\nabla_g\psi).
    \]
    Therefore, we obtain:
    \begin{equation*}
    \langle \nabla \varphi,\nabla \psi\rangle=g(\nabla_g\varphi,\nabla_g\psi)\qquad \meas\text{-a.e.\@ on }Y_+,
    \end{equation*}
    for all $\varphi,\psi\in \mathrm{Lip}(Y)$.
\end{rem}

We also need a technical Lemma relating the space of locally Sobolev functions to the space of functions whose weak derivatives w.r.t.\@ any chart are locally Sobolev.

\begin{lem}\label{lem:identification_RCD_W^12_VS_Riem_W^12}
    The space $W^{1,2}_{loc}(Y_+)$ (see Definition \ref{defn:localy_Sobolev}) coincides with the more classical space $H^1_{loc}(Y_+)$ (see Remark \ref{rem:H^1_loc}).
\end{lem}

\begin{proof}
    Assume that $\varphi\in W^{1,2}_{loc}(Y_+)$ and let $\eta\in \mathrm{Lip}_c(Y_+)$. By definition $\eta\varphi$ lies in $W^{1,2}(Y)$ and has compact support in $Y_+$. In particular, as a result of \cite{MondinoRyborz}*{Corollary 4.27}, we have $|D(\eta\varphi)|_w=|\nabla_g(\eta\varphi)|\in L^2_c(Y_+)$. Since $g$ is continuous and $\meas$ is equivalent to $\mathrm{dvol}_g$ on $Y_+$, then $\eta\varphi\in H^1_{loc}(Y_+)$. Replacing $\eta$ with $\mathrm{Lip}_c(Y_+)$-functions whose supports exhaust $Y_+$, we obtain $\varphi\in H^1_{loc}(Y_+)$.
    
    Conversely, let $\varphi\in H^1_{loc}(Y_+)$ and $\eta\in\mathrm{Lip}_c(Y_+)$. Then $|\nabla_g\varphi|\in L^2_{loc}(Y_+)$ and, thanks to Lemma \ref{lem: weak gradient vs lip}, we have $|\nabla_g\eta|\in L^{\infty}_c(Y_+)$. Therefore, since $|\nabla_g(\eta\varphi)|\le|\eta||\nabla_g\varphi|+|\varphi||\nabla_g\eta|$, then $|\nabla_g(\eta\varphi)|\in L^2_c(Y_+)$. Therefore, by \cite{MondinoRyborz}*{Corollary 4.27}, $|D(\eta\varphi)|_w$ lies in $L^2_c(Y_+)$, i.e., $\varphi$ lies in $W^{1,2}_{loc}(Y_+)$.
\end{proof}

Since $H$ is a connected, simply connected, and nilpotent Lie group, $\exp\colon \mathfrak{h}\to H$ is a diffeomorphism and we may identify $H$ with $\R^k$ equipped with coordinates $(x^1,\ldots,x^k)$. Using Remark \ref{rem:g(nabla u,nabla v)}, we obtain
\begin{equation*}\label{eq:g(nablaxi,nablaxj)}
  g(\nabla_gx^i,\nabla_gx^j) =  G^{ab}(r)X_a(x^i)X_b(x^j).  
\end{equation*}
In other words, we have:
\begin{equation}\label{eq:GvsGbar}
    G=A\overline{G}^{-1}A^{{T}},
\end{equation}
where $\overline{G}\coloneqq [g(\nabla_gx^i,\nabla_gx^j)]$ and $A\coloneqq [X_i(x^j)]$ is smooth and independent of $r$. The proof of Proposition \ref{prop:W^12G(r)} relies on the following lemma. 

\begin{lem}
\label{lem:nabla_gx^i|nabla_gx^j locally sobolev}
    For every $1\le i,j\le k$, we have $\overline{G}_{ij}=g(\nabla_gx^i,\nabla_gx^j)\in W^{1,2}_{loc}(Y_+)$.
\end{lem}

Let us explain how Proposition \ref{prop:W^12G(r)} follows from Lemma \ref{lem:nabla_gx^i|nabla_gx^j locally sobolev}.

\begin{proof}[Proof of Proposition \ref{prop:W^12G(r)}]
    First of all, we recall that $G\colon (0,\infty)\to \mathrm{GL}_k(\R)$ is continuous thanks to Lemma \ref{lem:g_r_cont}. Thus, it follows from \eqref{eq:GvsGbar} that $\overline{G}\colon Y_+\to \mathrm{GL}_k(\R)$ is continuous, which implies that $\overline{G}_{ij}\in L^{\infty}_{loc}(Y_+)$. As a result of Remark \ref{rem:loc_sobolev_algebra},  \eqref{eq:GvsGbar}, and Lemma \ref{lem:nabla_gx^i|nabla_gx^j locally sobolev} we have $G_{ij}(r)\in W^{1,2}_{loc}(Y_+)$. In particular, since $G_{ij}(r)$ only depends on $r$, we conclude that $G_{ij}(r)\in W^{1,2}_{loc}\left((0,
    \infty)\right)$ thanks to Lemma \ref{lem:identification_RCD_W^12_VS_Riem_W^12}.
    

\end{proof}

While our coordinate functions $\{x^i\}_{1\le i\le k}$ are smooth, they are (a priori) only locally Lipschitz on $Y_+$. However, we will see in our next lemma that, by multiplying them with a cut-off test function we will obtain elements of $D(\Delta)$. Such regularity will later ensure that $\langle \nabla(\chi x^i),\nabla(\chi x^j)\rangle$ is Sobolev on $Y$. From there, we will easily conclude that $g(\nabla_g x^i,\nabla_gx^j)$ is locally Sobolev on $Y_+$.

\begin{lem}\label{lem:chi_x^i_in_D(Delta)}
    For every $\chi\in \mathrm{TestF}(Y)$ with compact support in $Y_+$, we have $\chi x^i\in D(\Delta)$.
\end{lem}

\begin{proof}
    First, we observe that $\chi x^i\in \mathrm{Lip}(Y)\subset W^{1,2}(Y)$ since $\chi\in \mathrm{Lip}_c(Y_+)$ and $x^i\in\mathrm{Lip}_{loc}(Y_+)$ thanks to Remark \ref{rem:smooth_implies_Lip_loc}.

    Now, let $\varphi\in\mathrm{Lip}_c(Y)$ and let us show that
    \[
    \int_Y\langle \nabla\varphi,\nabla(\chi x^i)\rangle\di\mathfrak{m}=-\int_Y\varphi\psi\di\mathfrak{m}
    \]
    for some $\psi\in L^2(Y)$.
    
    Thanks to Remark \ref{rem:RCD<>toRiem<>}, the following holds $\meas$-almost everywhere:
    \begin{equation}\label{eq: chain rule chi v}
    \begin{aligned}
    \langle \nabla \varphi,\nabla (\chi x^i)\rangle
    &=g(\nabla_g\varphi,\nabla_g(\chi x^i)) \\
    &=
    \underbrace{g(\nabla_g(x^i\varphi),\nabla_g\chi)}_{\eqqcolon A}
    +
    \underbrace{g(\nabla_g(\chi\varphi),\nabla_g x^i)}_{\eqqcolon B}
    +\underbrace{\bigl(-2\varphi\,g(\nabla_g x^i,\nabla_g\chi)\bigr)}_{\eqqcolon C},
    \end{aligned}
\end{equation}
where each term is defined $\meas$-a.e.\@ thanks to Lemma \ref{lem: weak gradient vs lip}. Let us now deal with each of the terms $A$, $B$, and $C$ from \eqref{eq: chain rule chi v}.

For $A$, we fix a cut-off function $\psi\in \mathrm{Lip}_c(Y_+)$ equal to $1$ on $\Spt(\chi)$. As a result, we have $\psi x^i\varphi\in\mathrm{Lip}_c(Y)$ and the following holds $\meas$-a.e.\@ thanks to Remark \ref{rem:RCD<>toRiem<>}:
\[
A=g(\nabla_g(\psi x^i\varphi),\nabla_g\chi)=\langle \nabla(\psi x^i\varphi),\nabla\chi\rangle
\]
Since $\chi\in D(\Delta)$, we have:
    \[
    \int_Y A\di \mathfrak{m}=\int_Y\langle\nabla(\psi x^i\varphi),\nabla\chi\rangle\di \mathfrak{m}=-\int_Y\psi x^i\varphi\Delta\chi\di \mathfrak{m}=-\int_Y\varphi x^i\Delta\chi\di \mathfrak{m},
    \]
    using the locality of the Laplacian in the last equality, i.e., $\Delta \chi=0$ $\meas$-a.e.\@ outside of $\Spt(\chi)$. Observe that $\Delta \chi\in L^2_c(Y_+)$ since $\chi\in D(\Delta)$ has compact support in $Y_+$. In particular, since $x^i\in\mathrm{Lip}_{loc}(Y_+)$ thanks to Remark \ref{rem:smooth_implies_Lip_loc}, we have $x^i\Delta\chi\in L^2(Y)$.
    
    Thanks to Remark \ref{rem:g(nabla u,nabla v)}, the following holds $\meas$-almost everywhere:
    \[
    B=g(\nabla_g(\chi\varphi),\nabla_g x^i)=G^{ab}(r)X_a(\chi\varphi)X_b(x^i).
    \]
    We also recall that, thanks to Proposition \ref{prop:measCD(0,N)}, we have $\meas=\rho(r)\di r\wedge\Omega$. Thus, we obtain:
    \[
        \int_Y B\di \mathfrak{m}=\int_{r=0}^{\infty}\rho(r)G^{ab}(r)\left(\int_{H}X_a(\chi\varphi)X_b(x^i)\ \Omega\right)\di r.
    \]
    Now, observe that $\mathcal{L}_{X_a}(\Omega)=\mathrm{div}_{\Omega}(X_a)\Omega$, where $\mathrm{div}_{\Omega}(X_a)=-\mathrm{Tr}(\mathrm{ad}_{X_a})$ because $X_a$ is left-invariant. However, since $H$ is nilpotent, we have $\mathrm{Tr}(\mathrm{ad}_{X_a})=0$, which implies $\mathcal{L}_{X_a}(\Omega)=0$. As a result, we have:
    \[X_a(\chi\varphi)X_b(x^i)\ \Omega=\mathcal{L}_{X_a}((\chi\varphi)X_b(x^i)\ \Omega)-(\chi\varphi)X_a(X_b(x^i))\ \Omega.
    \]
    Moreover, by Stokes' Theorem, we have $\int_H\mathcal{L}_{X_a}((\chi\varphi)X_b(x^i)\ \Omega)=0$. Consequently, we obtain:
    \[
    \int_Y B\di \mathfrak{m}=-\int_Y\varphi(\underbrace{\chi G^{ab}(r)X_a\bigl(X_b(x^i)\bigr)}_{\eqqcolon\mu})\di\mathfrak{m},
    \]
    where $\mu\in L^2(Y)$ since $G^{ab}(r)\in C^0(Y)$, $\chi\in\mathrm{Lip}_c(Y_+)$, and $X_a\bigl(X_b(x^i)\bigr)\in\mathrm{Lip}_{loc}(Y_+)$.
    \smallskip

    \noindent For $C$, observe that:
    \[
    \lvert -2g(\nabla_g x^i,\nabla_g\chi)\rvert\le 2\lvert \nabla_g x^i\rvert\lvert\nabla_g\chi\rvert \le 2\mathrm{Lip}(\chi)\lvert \nabla_g x^i\rvert,
    \]
    thanks to Lemma \ref{lem: weak gradient vs lip}. As a result, we have $2g(\nabla_g x^i,\nabla_g\chi)\in L^{\infty}_c(Y_+)$.
    \smallskip

    Thanks to the previous three paragraph, we have $\chi x^i\in D(\Delta)$ and
    \[\Delta(\chi x^i)=2g(\nabla_g x^i,\nabla_g\chi)+x^i\Delta \chi+\mu\in L^2(Y).
    \]
\end{proof}

We may now prove Lemma \ref{lem:nabla_gx^i|nabla_gx^j locally sobolev}.

\begin{proof}[Proof of Lemma \ref{lem:nabla_gx^i|nabla_gx^j locally sobolev}]
    Given $1\le i,j\le k$ and $\chi\in \mathrm{TestF}(Y)$ with compact support in $Y_+$, we have $\chi x^i,\chi x^j\in H^{2,2}(Y)$ thanks to Lemma \ref{lem:chi_x^i_in_D(Delta)} and \cite{gigli_nonsmooth}*{Proposition 3.3.18}. Moreover, we recall that $\chi x^i,\chi x^j\in \mathrm{Lip}(Y)$ since $\chi\in \mathrm{Lip}_c(Y_+)$ and $x^i,x^j\in\mathrm{Lip}_{loc}(Y_+)$. As a result, $\chi x^i$ and $\chi x^j$ both have bounded gradient. According to \cite{gigli_nonsmooth}*{Proposition 3.3.22}, we have
    \[
    \langle \nabla (\chi x^i),\nabla (\chi x^j)\rangle\in W^{1,2}(Y),
    \]
    where we recall that $\langle \nabla (\chi x^i),\nabla (\chi x^j)\rangle = g(\nabla_g(\chi x^i),\nabla_g(\chi x^j))$ thanks to Remark \ref{rem:RCD<>toRiem<>}.
    
    To show $g(\nabla_g x^i,\nabla_g x^j)\in W^{1,2}_{loc}(Y^+)$, we need to show that, given any $\eta\in \mathrm{Lip}_c(Y_+)$, we have $\eta g(\nabla_g x^i,\nabla_g x^j)\in W^{1,2}(Y)$. However, given $\eta\in \mathrm{Lip}_c(Y_+)$, we may fix $\chi\in \mathrm{TestF}(Y)$ with compact support in $Y_+$ such that $\chi$ is constant equal to $1$ on $\Spt(\eta)$. In particular, we have
    \[
    \eta\; g(\nabla_g x^i,\nabla_g x^j)=\eta\; g(\nabla_g(\chi x^i),\nabla_g(\chi x^j))\in W^{1,2}(Y),
    \]
    since the product of a $W^{1,2}(Y)$ function with a $\mathrm{Lip}_c(Y)$ function is a $W^{1,2}(Y)$ function. This concludes the proof.
\end{proof}

We now conclude this section with item \eqref{item:distributional-Ricci} of Theorem \ref{thm:RCD+nilpotent}, which will play a crucial role when looking for a lower bound on $N$ in Section \ref{sec:Nbound}.

\begin{prop}\label{prop:sobolev_christoffel_and_V}
    The Christoffel symbols of $g$ lie in $L^2_{loc}(Y_+)$ and $\meas=e^{-V}\mathrm{dvol}_g$, where
    \[
    V(r)=\ln (\sqrt{\det G(r)})-\ln \rho(r)
    \]
    lies in $C^0\cap W^{1,2}_{loc}(Y_+)$. Moreover, the inequality
    \[
    \Ric_{\meas,N}\ge0
    \]
    holds on $Y_+$ in the sense of distributions (see Definition \ref{defn:distributional_BE}).
\end{prop}

\begin{proof}
    We recall that $Y_+$ is equipped with $g=\di r^2+g_{r}$. As a result of Lemma \ref{lem:g_r_cont} and Proposition \ref{prop:W^12G(r)}, $g$ is locally given as a positive definite matrix whose coefficients $G_{ij}$ and inverse coefficients $G^{ij}$ lie in $L^{\infty}_{loc}\cap H^1_{loc}(Y_+)$ (see Definition \ref{defn:distributional_BE} and Remark \ref{rem:H^1_loc}). Therefore, we may apply Koszul's formula \ref{eq:Koszul} to compute the Christoffel symbols of $g$ in the frame $\{\partial_r,X_1,\ldots,X_k\}$. First, we denote $S(r)$ the shape operator of $\{r\}\times H\subset Y_+$ defined by 
    \[
        g(S(r)X,Y)=g(\nabla_X\partial_r,Y),
    \]
    for smooth vector fields $X,Y$ in the {$H$-}orbit direction. By Koszul formula, we have
    \[
    g(S(r)X,Y)=\dfrac{1}{2} \partial_r g_r(X,Y),
    \]
    so $S(r)$ takes the matrix form
    \[
    S(r)=\dfrac{1}{2} G(r)^{-1} G'(r),{\text{ with the convention } S(r)X_j=S^{i}{}_jX_i}.
    \]
    Using Koszul's formula, we obtain
    \begin{equation}\label{eq:christoffel}
        \Gamma^r_{ij}=-\tfrac12 G'_{ij},
        \qquad
        \Gamma^i_{rj}=S^i{}_j,
        \qquad
        \Gamma^i_{rr}=\Gamma^r_{ri}=0,
        \qquad i,j\in\{1,\ldots, k\},
    \end{equation}
    while $\Gamma^\ell_{ij}$ ($i,j,\ell\in\{1,\ldots, k\}$) are the Christoffel symbols of the left-invariant metric $g_r$ on $H$ in the frame $\{X_1,\ldots,X_k\}$. As a result of Proposition \ref{prop:W^12G(r)}, the Christoffel symbols of $g$ lie in $L^2_{loc}(Y_+)$.

    Thanks to item \eqref{item:measure} of Theorem \ref{thm:RCD+nilpotent}, $Y_+$ is equipped with the measure
    \[
    \meas=\rho(r)\di r\wedge\Omega=e^{-V}\mathrm{dvol}_g,\ V(r)=\ln (\sqrt{\det G(r)})-\ln \rho(r)\in C^0\cap W^{1,2}_{loc}((0,\infty)).
    \]
    In particular, $V(r)$ lies in $C^0\cap W^{1,2}_{loc}(Y_+)$ thanks to Lemma \ref{lem:identification_RCD_W^12_VS_Riem_W^12}.

    Therefore, thanks to the two paragraphs above and since $(Y,d)$ is the metric completion of $(Y_+,d_g)$ (see item \ref{item:continuous+sobolev} of Theorem \ref{thm:CD-to-Distribution}), we can conclude using Theorem \ref{thm:CD-to-Distribution}.
\end{proof}

\begin{rem}
    The Christoffel symbols $\Gamma_{ij}^l$ without $r$ may involve terms coming from the possibly non-zero Lie brackets $[X_i,X_j]$, but they will not be used in our later computations.
\end{rem}

\section{Estimate dimension $N$ by Hausdorff dimension}\label{sec:Nbound}
In this section, our aim is to prove Theorem \ref{mainthm:RCD_N_bdd}. Throughout this section, we always assume that $(Y,y,d,\meas)$ is an $\RCD(0,N)$ space and $H$ is a connected, simply connected, and nilpotent Lie group satisfying the assumptions of Theorem \ref{mainthm:RCD_N_bdd}. In particular, $(Y,y,d,\meas)$ satisfies the assumptions of Theorem \ref{thm:RCD+nilpotent} and items \eqref{item:regular_set} to \eqref{item:distributional-Ricci} hold. We will re-use the notations of Sections \ref{sec:measure_CD(0,N)} and \ref{sec:loc_sobolev_metric}, i.e., $\{X_1,...,X_k\}$ denotes a fixed left-invariant frame on $H$ with dual coframe $\{\theta^1,\ldots,\theta^k\}$ and associated left-invariant Haar measure $\Omega=\theta^1\wedge...\wedge \theta^k$. As a result of Theorem \ref{thm:RCD+nilpotent}, the regular part $Y_+$ comes from a weighted Riemannian metric
$$g=\di r^2 + g_r,\quad \meas=\rho(r)\di r\wedge \Omega,$$
where $\{g_r\}$ is a family of left-invariant Riemannian metrics on $H$ that is continuous and locally $W^{1,2}$ in $r$, and $\rho$ is a $\mathrm{CD}(0,N)$ density function. In particular, the $k\times k$ matrix $G(r)$ of $g_r$ in the basis $\{X_1,...,X_k\}$ satisfies
\begin{equation}\label{nota:metric_matrix}
G(r)=\left[ G_{ij}(r) \right]\defeq\left[ g_r(X_i,X_j) \right]\in W^{1,2}_{loc}((0,\infty),\mathrm{Sym}_k(\R)).
\end{equation}
All the metric measure information of $Y$ is encoded in $\rho(r)$ and the one-parameter family $\{G(r)\}_{r>0}$.

In Section \ref{sec:curv_ineq}, we first derive an inequality for the distributional Bakry--Emery Ricci curvature in terms of the ordered eigenvalues $\lambda_1(r)\le\dots \le \lambda_k(r)$ of $G(r)$ (Proposition \ref{prop:Ric_rr}). We remark that this inequality is distinct from the standard Ricci curvature inequality used in Laplacian and volume comparison (see Remark \ref{rem:Riccati_comparison}). Our curvature inequality involves a trace inequality, which we prove in Appendix \ref{appx:perturbation_matrices}. Then we apply this curvature inequality to our situation and use item \eqref{item:distributional-Ricci} of Theorem \ref{thm:RCD+nilpotent} to obtain an inequality in terms of the ordered eigenvalues $\lambda_1(r)\le\dots \le \lambda_k(r)$, density function $\rho(r)$, and dimension $N$ (Lemma \ref{lem:Nbound}).

In Section \ref{sec:eigenvalues_by_Haus_dim}, we estimate the smallest and largest eigenvalues of $G(r)$ as $r\to 0$ under the assumptions of Theorem \ref{mainthm:RCD_N_bdd}.

Lastly, we complete the proof of Theorem \ref{mainthm:RCD_N_bdd} in Section \ref{sec:prove_N_bdd}.

\subsection{A curvature inequality}\label{sec:curv_ineq}

Let $$0<\lambda_1(r)\le \ldots\le \lambda_k(r)$$ be the ordered eigenvalues of $G(r)$. It follows from Weyl's perturbation theorem (see Theorem \ref{thm:Weyl}) that 
$$\lambda_j(r)\in C^0\cap W^{1,2}_{loc}\left((0,\infty)\right), \quad j=1,\ldots,k.$$ 
We recall from the proof of Proposition \ref{prop:sobolev_christoffel_and_V} that the shape operator $S(r)$ of $\{r\}\times H\subset Y_+$ takes the matrix form
\[
S(r)=\dfrac{1}{2} G(r)^{-1} G'(r),{\text{ with the convention } S(r)X_j=S^{i}{}_jX_i}
\]
We first estimate the Ricci curvature in terms of the ordered eigenvalues of $G(r)$. We state a general version when $N=\infty$ here because it may have independent interest and broader applications. A more specific version will be given in Lemma \ref{lem:Nbound}.

\begin{prop}\label{prop:Ric_rr}
The following inequality holds in the sense of distribution:
   $$ \Ric_{\meas,\infty}(\partial_r,\partial_r) \le -\dfrac{1}{4}\sum_{j=1}^k \left(\dfrac{\lambda'_j}{\lambda_j}\right)^2 - \left( \frac{\rho'}{\rho} \right)'.$$
\end{prop}

\begin{proof}
Recall that, by Proposition \ref{prop:sobolev_christoffel_and_V}, $\meas = e^{-V}\mathrm{dvol_g}$, where $V(r)=\ln (\sqrt{\det G(r)})-\ln \rho(r)$ lies in $C^0\cap W^{1,2}_{loc}\left((0,\infty)\right)$. 
Then it follows from {\ref{eq:christoffel}}, \eqref{eq:distributional_Ric} and \eqref{eq:distributional_Hess} that
\begin{align*}
    \Ric(\partial_r,\partial_r)&=\sum_{i,j=1}^k G^{ij}g(X_i,\mathrm{Rm}(X_j,\partial_r)\partial_r)\\&=-\sum_{i,j=1}^k G^{ij}g(X_i,\nabla_{\partial_r}\nabla_{X_j}\partial_r)\\&=-\sum_{i,j,\ell=1}^k G^{ij}g\!\left(X_i,\bigl((S')^\ell{}_j+S^2)^\ell{}_j\bigr)X_\ell\right)\\&=-\sum_{i,j,\ell=1}^k G^{ij}G_{i\ell}\bigl((S')^\ell{}_j+(S^2)^\ell{}_j\bigr)\\&=-\tr(S')-\tr(S^2).\\
  \Hess V(\partial_r,\partial_r)&= \left(\ln \sqrt{\det G(r)}\right)''-\left( \frac{\rho(r)'}{\rho(r)} \right)'
\end{align*}
 Notice that $G(r)\in W^{1,2}_{loc}$, by the Sobolev Chain rule $(\ln {\det G})'= \tr (G^{-1}G')= 2\tr S$. It follows that 
 \[
 \tr S(r)=(\ln {\sqrt{\det G}})'.
 \]
 Together with the trace inequality Proposition \ref{prop:Trace-inequality} we have 
 \begin{align*}
      \Ric_{\meas,\infty}(\partial_r,\partial_r)&= \Ric(\partial_r,\partial_r)+  \Hess V(\partial_r,\partial_r)\\
      &=\underbrace{-(\tr S(r))'+\left(\ln \sqrt{\det G(r)}\right)''}_{=0}-\tr(S(r)^2)-\left( \frac{\rho(r)'}{\rho(r)} \right)'\\
      &\le -\dfrac{1}{4}\sum_{j=1}^k \left(\dfrac{\lambda'_j}{\lambda_j}\right)^2 - \left( \frac{\rho'}{\rho} \right)',
 \end{align*}
 as desired.
\end{proof}

  \begin{rem}\label{rem:Riccati_comparison}
    In the Riccati equation
    \[\Ric(\partial_r,\partial_r)=-(\tr S(r))'-\tr(S(r)^2),\]
    the standard Laplacian and volume element comparison use the Cauchy--Schwarz inequality 
    \[
    \tr (S(r)^2)\ge \dfrac{1}{k}\cdot \tr (S(r))^2.
    \]
    Following this approach, we can estimate $N$ in terms of the eigenvalues of $S(r)$ and get the weaker bound $N\ge 28/3=9.33\ldots$, assuming $k=3$ and $s=2$. Our method uses the eigenvalues of $G(r)$ via Proposition \ref{prop:Trace-inequality} and gives the sharp estimate $N\ge 12$; also see Example \ref{exmp:N_bdd_sharp}. 
\end{rem}

\begin{rem}
    {Any $L^1_{loc}(0,\infty)$-function $f$ gives rise to a signed Radon measure $f(r)\di r$ on $(0,\infty)$, which can naturally be seen as a distribution on $(0,\infty)$. As a result, we will identify $L^1_{loc}(0,\infty)$-functions with their induced Radon measure and distribution, and we will identify Radon measures with their induced distribution.}
\end{rem}

To simplify the notation, we denote 
$$J(r)\coloneqq\sqrt{\det G(r)},\quad b_j\coloneqq\dfrac{\lambda'_j}{2\lambda_j}=\dfrac{1}{2}(\ln \lambda_j(r))',\quad b\coloneqq\dfrac{J'}{J}.$$
Then,
$$
\ln J(r)=\dfrac{1}{2}\ln\det G(r) = \dfrac{1}{2}\sum_{j=1}^k \ln \lambda_j(r),\quad b=(\ln J)'=\sum_{j=1}^k b_j.
$$
Since $G(r)$ is positive definite, we infer that $J$ is positive and lies in $C^0\cap W^{1,2}_{loc}\left((0,\infty)\right)$, meanwhile $\lambda'_j\in L^2_{loc}\left((0,\infty)\right) $ and $1/\lambda_j\in L^\infty_{loc}\left((0,\infty)\right)$. It follows $b_j,b\in L^2_{loc}\left((0,\infty)\right)$. In these notations, Proposition \ref{prop:Trace-inequality} can be written as an inequality between distributions. 
\begin{equation}\label{eq:trace-ineq-distribution}
    \tr(S(r)^2)\ge \sum_{j=1}^{k} b_j^2.
\end{equation}
Here both sides are in $L^1_{loc}\left((0,\infty)\right)$; therefore, the distributional inequality is equivalent to an a.e.\@ inequality between $L^1_{loc}\left((0,\infty)\right)$-functions.
We also write $u\coloneqq\rho^{1/(N-1)}$, which is concave, and let
$$
a\defeq \frac{u'}{u}=\frac{1}{N-1}\frac{\rho'}{\rho}.
$$
In particular, since $a$ is the logarithmic derivative of a concave function, it {lies in $\mathrm{BV}\left((0,\infty)\right)$}, admits a pointwise representative, and its {BV-}derivative satisfies {the following thanks to} the product rule for $\mathrm{BV}$-functions
\[
a'=\frac{u''}{u}-a^2,
\]
{where $a'$ and $u''$ are signed Radon measures on $(0,\infty)$.}

By direct computation we have the following inequality.

\begin{lem}\label{lem:Nbound} With the notations above, we have
   \begin{equation}\label{eq:Nbound}
       \sum_{j=1}^k (b_j-a)^2+\dfrac{1}{N-k-1}\left(b-ka\right)^2\le -(N-1)\dfrac{u''}{u},
   \end{equation}
   {in the sense of signed Radon measures on $(0,\infty)$.}
\end{lem}

\begin{proof}
 Rewrite Proposition \ref{prop:Ric_rr} in the current notations as 
 \[
 \Ric_{\m,\infty}(\partial_r,\partial_r)\le - \sum_{j=1}^k b_j^2-(N-1)a'.
 \]
 Direct computation gives
 \begin{align*}
(\nabla V \otimes \nabla V)(\partial_r,\partial_r)=V'(r)^2=\left( \frac{J'}{J} - \frac{\rho'}{\rho} \right)^2=\left(b-(N-1)a\right)^2.
\end{align*}
Applying item \eqref{item:distributional-Ricci} of Theorem \ref{thm:RCD+nilpotent} to the smooth vector field $\partial_r$ gives the following inequality in the distributional sense
\begin{equation}\label{eq:Ric}
    \begin{split}
       0&\le \Ric_{\m,N}(\partial_r,\partial_r)=\Ric_{\m,\infty}(\partial_r,\partial_r)-\frac{ V'(r)^2 }{N-n},\quad n = \dim(Y_+)=k+1,\\
       & \le -\sum_{j=1}^k b^2_j -(N-1)a'-\dfrac{1}{N-k-1}\left( b -(N-1)a\right)^2\\
       &=-(N-1)\dfrac{u''}{u}+(N-1)a^2-\sum_{j=1}^k b_j^2 -\dfrac{1}{N-k-1}\left( b -(N-1)a\right)^2.
    \end{split}
\end{equation}
After this simplification, we see that only the $u''$ term involves a Radon measure, while all other terms are $L^1_{loc}\left((0,\infty)\right)$-functions. {Therefore, identifying $L^1_{loc}\left((0,\infty)\right)$-functions with their induced signed Radon measure, the distributional inequality \eqref{eq:Ric} is equivalent to an inequality between signed Radon measures. We rearrange the terms as follows:}
\begin{align*}
\left( b -(N-1)a\right)^2 &= \left( (b-ka) -(N-k-1)a\right)^2\\
&=\left(b-ka\right)^2-2(N-k-1)a(b-ka)+(N-k-1)^2a^2.\\
\sum_{j=1}^k b_j^2&=\sum_{j=1}^k (b_j-a+a)^2=\sum_{j=1}^k (b_j-a)^2 +2a(b-ka) + ka^2.
\end{align*}
Plugging the above identities back into \eqref{eq:Ric} gives the desired inequality.
\end{proof}

\subsection{Estimate eigenvalues of $G(r)$ by Hausdorff dimension}\label{sec:eigenvalues_by_Haus_dim}

In this section, we obtain lower bounds for the eigenvalues of $G(r)$ from the geometric assumptions in Theorem \ref{mainthm:RCD_N_bdd}. 

We first prove a lower bound on the smallest eigenvalue of $G(r)$ under assumption (1) of Theorem \ref{mainthm:RCD_N_bdd}, or equivalently, the assumptions in Theorem \ref{thm:RCD+nilpotent}.

\begin{lem}\label{lem:eigenvalue-lower-bound}
    Under the assumptions of Theorem \ref{thm:RCD+nilpotent} and notation \eqref{nota:metric_matrix}, let $\lambda_1(r)>0$ be the smallest eigenvalue of $G(r)$. Then $\liminf_{r\to 0^+} \lambda_1(r)>0.$
\end{lem}

\begin{proof}
   Suppose otherwise, then there is a sequence $r_i\to 0$ such that $\lambda_1(r_i) \to 0$. For each $i$, we choose $$V_i=\sum_{j=1}^k x_j(r_i)X_j \in \mathfrak{h}$$
such that
$$\sum_{j=1}^k x_j^2(r_i)=1,\quad g_{r_i}(V_i,V_i)=\lambda_1(r_i).$$
Let $h_i = \exp(V_i)\in H$. After passing to a subsequence if necessary, we can assume $$(x_1(r_i),...,x_k(r_i))\to (x^o_1,...,x_k^o)\in S^{k-1}.$$
Let $V^o= \sum x_j^o X_j \in \mathfrak{h}$ and $h=\exp(V^o)\in H$. Then $h_i\to h\neq \mathrm{id}$ on $H$ by construction.

For each $i$, let us estimate the distance between $y$ and $h_i y$ below, where we recall that $y=(0,\mathrm{id})\in [0,\infty)\times H$. We draw a horizontal segment from $y$ to $(r_i,\mathrm{id})$. Then we follow the curve 
$$\sigma_i: [0,1]\to Y,\quad t\mapsto \exp(tV_i)\cdot (r_i,\mathrm{id})=(r_i,\exp(tV_i)),$$
which ends at $(r_i,h_i)$. Lastly, we draw a horizontal segment from $(r_i,h_i)$ to $(0,h_i)$. Then we can estimate $d(y,h_i y)$ by the length of the concatenated curve
$$d(y,h_iy)\le 2r_i+ \lambda_1(r_i)^{1/2}\to 0.$$
as $i\to\infty$. This is a desired contradiction to $h_i\to h\neq \mathrm{id}$ and thus completes the proof.
\end{proof}

Next, we assume assumption (2) of Theorem \ref{mainthm:RCD_N_bdd} and prove a lower bound on the blow-up rate of the largest eigenvalue of 
$G(r)$ as $r\to 0^+$.

\begin{lem}\label{lem:largest_eignvalue_bound}
   Under the assumptions of Theorem \ref{thm:RCD+nilpotent} and notation \eqref{nota:metric_matrix}, let $\lambda_k(r)>0$ be the largest eigenvalue of $G(r)$. Suppose that there is an $\R$-subgroup $L$ of $H$ such that the orbit $Ly$ has Hausdorff dimension $\dim_H(Ly)\ge s>1$. Then for any $\epsilon>0$, there is $\delta>0$ such that
   $$\lambda_k(r)\ge (r^{-\alpha+\epsilon})^2$$
   for all $r\in (0,\delta)$, where $\alpha=s-1$.
\end{lem}

\begin{proof}
  We argue by contradiction and suppose the contrary. Then there are some small $\epsilon>0$ and a sequence $r_i\to 0$ such that
$$\lambda_k(r_i)\le (r_i^{-\alpha+\epsilon})^2.$$
We shall show that any $\R$-orbit at $y$ has Hausdorff dimension at most $1+\alpha-\epsilon$, which is a desired contradiction to the assumption.

For any vector $V=\sum x_jX_j\in \mathfrak{h}$, where $\sum x_j^2=1$, we have
$$g_{r_i}(V,V)^{1/2}\le \lambda_k(r_i)^{1/2} \le r_i^{-\alpha+\epsilon}.$$
For convenience, we denote $h=\exp(V)$ and $th=\exp(tV)$. For any $t\in(0,1)$, we estimate the distance between $d(y,(th)y)$ using the same method as in the proof of Lemma \ref{lem:eigenvalue-lower-bound}:
$$d(y,(th)y)\le 2r_i + t\cdot  r_i^{-\alpha+\epsilon}.$$
Let $l_i = r_i^{1+\alpha-\epsilon}$, then we obtain
$$\sup_{t\in [0,l_i]} d(y,(th)y)\le 3r_i = 3 l_i^{\frac{1}{1+\alpha-\epsilon}}. $$
Then it follows from the definition of Hausdorff dimension that the corresponding $\R$-orbit at $y$ 
$\{ (th)\cdot y\mid t\in \R \}$ has Hausdorff dimension at most $1+\alpha-\epsilon$; a contradiction to the assumption.
\end{proof}

\subsection{Proof of Theorem B}\label{sec:prove_N_bdd}

We are now in a position to finish the proof of Theorem \ref{mainthm:RCD_N_bdd} by using Lemma \ref{lem:Nbound} and the eigenvalue estimates in Section \ref{sec:eigenvalues_by_Haus_dim}.

\begin{proof}[Proof of Theorem \ref{mainthm:RCD_N_bdd}]
Recall that in Lemma \ref{lem:Nbound} we obtained the inequality 
\begin{equation}\label{eq:intermediate-Nbound}
       \sum_j (b_j-a)^2+\dfrac{1}{N-k-1}\left(b-ka\right)^2\le -(N-1)\dfrac{u''}{u},
   \end{equation}
where 
\[
a=\frac{u'}{u}, \quad u=\rho^{\frac{1}{N-1}},\quad  b_j\coloneqq\dfrac{\lambda'_j}{2\lambda_j}=\dfrac{1}{2}(\ln \lambda_j(r))', \quad b=\sum_{j=1}^k b_j.
\]
To simplify the notations further, we denote
$$ h(r)= \lambda_k^{1/2}(r) r^\alpha,\quad  p_j(r)=rb_j(r), j=1,...,k-1,\quad p_k(r)=r\frac{h'(r)}{h(r)}=rb_k+\alpha, \quad q(r)=r\dfrac{u'(r)}{u(r)}.$$
Notice that $u$ is concave and non-negative, we have
$$u'(r)\le \dfrac{u(r)-u(0+)}{r}\le \dfrac{u(r)}{r}.$$
Hence $0\le q(r)\le 1$ for all $r>0$. We can use the new notations to rewrite the right-hand side of \eqref{eq:intermediate-Nbound}, with a change of variable $r\defeq e^{-t}$, as follows 
$$-(N-1)\dfrac{u''}{u}=\dfrac{N-1}{r^2}\left(q(1-q)-rq'(r)\right)=\dfrac{N-1}{r^2}\left(q(1-q)+\dfrac{\di q}{\di t}\right),$$ 
{where $\dfrac{\di q}{\di t}$ denotes the $\mathrm{BV}$-derivative of $t\mapsto q(e^{-t})$ and is a signed Radon measure.} Next, we handle the left-hand side of \eqref{eq:intermediate-Nbound}. For $j=1,...,k-1$,
$$b_j-a=\dfrac{1}{r}(p_j-q);$$
for $j=k$,
$$b_k-a =\frac{h'}{h}-\frac{\alpha}{r}-\frac{u'}{u}=\dfrac{1}{r}(p_k-\alpha-q).$$
Also,
$$b-ka=\sum_{j=1}^k (b_j-a)=\dfrac{1}{r}\left(\sum_{j=1}^k p_j -\alpha -kq \right).$$
Plugging each term into \eqref{eq:intermediate-Nbound} and canceling $1/r^2$, we see that \eqref{eq:intermediate-Nbound} now takes the form
\begin{equation}\label{eq:Ric_to_pq}
\sum_{j=1}^{k-1} (q-p_j)^2 + (q+\alpha-p_k)^2 +\dfrac{\left(kq+\alpha-\sum_{j=1}^k p_j\right)^2}{N-k-1} \le (N-1)\left(q(1-q)+\dfrac{\di q}{\di t}\right).
\end{equation}
Given $T>0$, we write $[f]_T\coloneqq\fint_0^Tf\di t$ for the average of a function $f\in L^1([0,T])$. We now estimate the integral of the right-hand side of \eqref{eq:Ric_to_pq}. Since $\phi(q)\defeq q(1-q)$ is concave on $[0,1]$, Jensen's inequality yields
\begin{align*}
{\left(q(1-q)\di r + \frac{\di q}{\di t}\right)\left([0,T]\right)}&={\left(\int_0^T q(1-q)\di r\right) + \frac{\di q}{\di t}\left([0,T]\right)}\\ 
&=  T\fint_0^T \phi(q) \di t  +q(e^{-T})-q(1) \\
&\le T \phi\left( \fint_0^T q\di t \right)+1\\
&= T [q]_T \left(1-[q]_T\right) +1.
\end{align*}
Next, we handle the left-hand side of \eqref{eq:Ric_to_pq}. By Cauchy--Schwarz inequality,
$$\int_0^T (q+\alpha-p_k)^2 \di t \ge \dfrac{1}{T}\left(\int_0^T (q+\alpha-p_k) \di t \right)^2=T\left(\alpha+[q]_T -[p_k]_T\right)^2.$$
Similarly,
\begin{align*}
\int_0^T \left(q- p_j\right)^2\di t  &\ge T\left( [q]_T -[p_j]_T \right)^2  \\
\int_0^T \left(kq+\alpha-\sum_{j=1}^k p_j\right)^2\di t&\ge T\left( \alpha +k[q]_T -\sum_{j=1}^k [p_j]_T \right)^2.
\end{align*}
For any $\epsilon\in (0,\alpha/10)$, by Lemmas \ref{lem:eigenvalue-lower-bound} and \ref{lem:largest_eignvalue_bound}
$$\lambda_j(r)\ge\eta >0\ \text{ for }  j=1,...,k,\quad h(r)= \lambda_k^{1/2}(r) r^\alpha\ge r^\epsilon$$ 
for some $\eta>0$ and all $r$ small enough. We note that
$$p_k(r)=r\dfrac{h'(r)}{h(r)}=-\dfrac{\di}{\di t}\ln (h(e^{-t})),$$
hence for $T$ large
$$-[p_k]_T=-\fint_0^T p_k \di t =\dfrac{1}{T}\left[\ln(h(e^{-T}))-\ln(h(1))\right]\ge -\epsilon+\dfrac{c_k}{T},$$
where $c_k=-\ln (h(1))$. Similarly, for $j=1,...,k-1$, 
$$-[p_j]_T =\dfrac{1}{2T}\left[\ln(\lambda_j(e^{-T}))-\ln(\lambda_j(1))\right]\ge \dfrac{c_j}{2T}\to 0$$
as $T\to\infty$, where $c_j=\ln \eta-\ln(\lambda_j(1))$.
We put these estimates into the integrated version of \eqref{eq:Ric_to_pq}, divide by $T$ on both sides, then let $T\to\infty$. It follows that
\begin{equation}\label{eq:L}
    (k-1)L^2+(\alpha+L-\epsilon)^2+\dfrac{1}{N-k-1}(\alpha+kL-\epsilon)^2\le (N-1)L(1-L),
\end{equation}
where $L\in [0,1]$ is any accumulation point of $[q]_T$ as $T\to\infty$. 

Lastly, we obtain the desired estimate of $N$ from \eqref{eq:L}. We rearrange \eqref{eq:L} to be a quadratic polynomial of $L$, then we have
\[
(N-1)^2L^2+(N-1)(2(\alpha-\epsilon)+k+1-N)L+(\alpha-\epsilon)^2(N-k)\le 0
\]
for some value $L\in [0,1]$. The leading coefficient is positive, so the discriminant must be nonnegative. Hence
\[
\Delta
=
(N-1)^2(2(\alpha-\epsilon)+k+1-N)^2
-
4(N-1)^2(\alpha-\epsilon)^2(N-k)
\ge 0.
\]
Canceling $(N-1)^2$ gives 
\[
(2(\alpha-\epsilon)+k+1-N)^2
-4(\alpha-\epsilon)^2-4(\alpha-\epsilon)^2(N-k-1) \ge0.
\]
 Writing the difference of squares as a product, we have
\[
(N-4(\alpha-\epsilon)-k-1)(N-k-1)-4(\alpha-\epsilon)^2(N-k-1)\ge0.
\]
Since $N-k-1>0$ we cancel this term to get
\begin{equation*}
    N\ge 4(\alpha-\epsilon)^2+4(\alpha-\epsilon)+1+k.
\end{equation*}
Recall that $s=\alpha+1$. Letting $\epsilon\to 0$, we obtain the desired inequality.
\end{proof}

\section{Application to fundamental groups}\label{sec:pi1}
In this section, we use Theorem \ref{mainthm:RCD_N_bdd} and the results from \cite{Pan26}, which we collected in Section \ref{sec:equiv_asym_cones} as Theorems \ref{thm:topol_dim} and \ref{thm:nil_dim}, to prove Theorem \ref{mainthm:nil_dim}. We start with an algebraic lemma about nilpotent Lie groups.

\begin{lem}\label{lem:nil_torsion_subgp}
   Let $G$ be a nilpotent Lie group with finitely many components. The closure of $\mathrm{Tor}(G)$ is a normal and compact subgroup of $G$.
\end{lem}

\begin{proof}
   The proof is standard in group theory. We include the proof here for readers' convenience.
   
   We first recall that, as a nilpotent group, the torsion elements in $G$ form a normal subgroup $T\coloneqq\mathrm{Tor}(G)$ of $G$. Let $G_0$ be the component subgroup of $G$ containing the identity element and let $\pi:\widetilde{G}_0\to G_0$ be the universal covering morphism of $G_0$. Then $\Gamma\coloneqq\ker(\pi)$ is a discrete and central subgroup of $\widetilde{G}_0$. Hence $\Gamma$ is contained in $\exp(\mathfrak{z})$, where $\exp:\mathfrak{g}\to \widetilde{G}_0$ is a diffeomorphism and $\mathfrak{z}$ is the center of the Lie algebra $\mathfrak{g}$ (see \cite{HN_book}*{Corollary 11.2.7 and Corollary 11.2.9}). Let $V\coloneqq\mathrm{span}_{\R} \{\exp^{-1}(\Gamma)\}$. Then $V$ is a subspace of $\mathfrak{z}$ and $\exp(V)/\Gamma=\pi(\exp V)$  provides the unique maximal torus in $G_0$. Now, let $g\in G_0$ be a torsion element with $g^m=\mathrm{id}$. Let $\tilde{g}=\exp(X)\in \widetilde{G}_0$ be a lift of $g$, where $X\in \mathfrak{g}$. Then
   $$\tilde{g}^m=\exp(mX)\in \ker(\pi)=\Gamma \subseteq \exp(V).$$
   Hence $X\in V$. As a consequence, $g=\pi(\exp(X))$ lies in the maximal torus of $G_0$. This verifies that the closure $\overline{T\cap G_0}$ coincides with the maximal torus of $G_0$; in particular, it is compact.

   Because $G$ has finitely many components, we can choose $t_1,...,t_k\in T$ in distinct components of $G$ that intersect $T$. Then
   $$\overline{T}=\bigcup_{i=1}^k t_i\cdot \overline{T\cap G_0}$$
   and thus $\overline{T}$ is compact.
\end{proof}

\begin{lem}\label{lem:dim_rank_ineq}
   Let $M$ be an open manifold with $\Ric\ge 0$ and let $\Lambda$ be a finitely generated torsion-free nilpotent group that acts on $M$ by covering group action. Then for any equivariant asymptotic cone $(Y,y,G)$ of $(M,\Lambda)$, the limit orbit $Gy$ has topological dimension at least $\mathrm{rank}(\Lambda)$.
\end{lem}

\begin{proof}
   The proof is almost identical to \cite{PanYe}*{Lemma 3.1}, where abelian groups were considered. The torsion-free nilpotent group $\Lambda$ admits a series of normal subgroups
   $$ \{\mathrm{id}\}=\Lambda^0 \triangleleft \Lambda^1 \triangleleft \dots \triangleleft \Lambda^k=\Lambda$$
   such that $k=\mathrm{rank}(\Lambda)$ and each $\Lambda^{j+1}/\Lambda^j$ is isomorphic to $\Z$ for $j=0,\dots,k-1$. For any sequence $r_i\to\infty$, we consider the equivariant asymptotic cone
   $$(r_i^{-1}M,p,\Lambda^1 \triangleleft \dots \triangleleft \Lambda^k)\overset{GH}\longrightarrow (Y,y,G^1\triangleleft \dots \triangleleft G^k).$$
   By the same argument in \cite{PanYe}*{Lemma 3.1}, $G^{j+1}/G^j$ contains a closed $\R$-subgroup for each $j=0,\dots,k-1$. Hence $Gy=G^{k}y$ has topological dimensional at least $k$.
\end{proof}

We prove Theorem \ref{mainthm:nil_dim} below.

\begin{proof}[Proof of Theorem \ref{mainthm:nil_dim}]
   We denote $\Gamma\coloneqq\pi_1(M)$. By a result of Sormani \cite{Sor00b}, nonnegative Ricci curvature and sublinear diameter growth together imply that $\Gamma$ is finitely generated. By the work of Milnor \cite{Milnor68} and Gromov \cite{Gromov81}, $\Gamma$ is almost nilpotent. Hence we can choose a torsion-free nilpotent normal subgroup $\Lambda$ of $\Gamma$ with finite index. We denote the rank of $\Lambda$ by $k_0$ and the step of $\Lambda$ by $s_0$. We remark that $k_0$ and $s_0$ are independent of the choice of $\Lambda$. By assumption, we have $k_0\ge k$ and $s_0\ge s\ge 2.$

   We first show that $n\ge k+2$. In fact, the torsion-free nilpotent group $\Lambda$ of rank $\ge k$ and step $\ge 2$ has polynomial growth of degree at least $k+1$; see \cites{Bass72,Guiv73}. Together with a volume argument by Anderson \cite{Anderson90}*{Theorem 1.1} and the fact that $M$ has at least linear volume growth \cites{Calabi75,Yau76}, it follows that the universal cover $\widetilde{M}$ has at least polynomial volume growth of degree $\ge k+2$. Hence $n\ge k+2$ by volume comparison.

   Below, we assume $n\ge k+2$. Let $\overline{M}=\widetilde{M}/\Lambda$ be an intermediate normal covering space over $M$. Because $\Lambda$ has finite index in $\Gamma$ and $M$ has sublinear diameter growth, $\overline{M}$ has a unique asymptotic cone as a ray $([0,\infty),0)$. By Proposition \ref{prop:E_not_1/2}, Theorem \ref{thm:topol_dim}, and Theorem \ref{thm:nil_dim}, there is an asymptotic cone of $(\widetilde{M},\Lambda)$,
   \begin{equation*}
\begin{CD}
(r_i^{-1} \widetilde{M},\tilde{p},\Lambda) @>GH>> (Y,y,G) \\
	@VV\pi V @VV \pi V\\
	(r_i^{-1} \overline{M},\bar{p}) @>GH>> ([0,\infty),0)=(Y/G,\bar{y})
\end{CD}
\end{equation*}
such that the following hold.\\
(1) $(Y,y)$, equipped with a limit renormalized measure, is $\RCD(0,n)$ with $n\ge k+2$ (see Remark \ref{rem:RCD_pmGH_stability});\\
(2) the limit group $G$ is nilpotent and acts on $Y$ by measure-preserving isometries;\\
(3) the orbit $Gy$ is connected and simply connected with $\dim(Gy) \ge k_0 \ge k$ (Theorem \ref{thm:topol_dim}(1) and Lemma \ref{lem:dim_rank_ineq});\\
(4) any compact subgroup of $G$ fixes $y$ (Theorem \ref{thm:topol_dim}(2));\\
(5) $G$ has at most finitely many components (Theorem \ref{thm:topol_dim}(3));\\
(6) there is a closed $\R$-subgroup $L\leq G$ such that the orbit $Ly$ has Hausdorff dimension at least $s_0\ge s\ge 2$ (Theorem \ref{thm:nil_dim}).\\
We note that the $G$-action on $Y$ may have a non-trivial isotropy subgroup at $y$ so we cannot directly apply Theorem \ref{mainthm:RCD_N_bdd} yet.

Let $K$ be the closure of the torsion subgroup of $G$. We claim that $K$ coincides with the isotropy subgroup of $G$ at $y$. By properties (2,5) above and Lemma \ref{lem:nil_torsion_subgp}, $K$ is compact. Thus property (4) implies that $K$-action fixes $y$. Conversely, let $g\in G$ fix $y$, and let us show that $g\in K$. If $g$ has finite order, then it is clear that $g\in K$. In general, as $g$ fixes $y$, the subgroup $\overline{\langle g \rangle}$ also fixes $y$ and thus is compact. Because $G$ has at most finitely many components by property (5), there is some $m\in \mathbb{N}$ such that $g^m \in G_0$. Now that $g^m$ is contained in the unique maximal torus $\mathbb{T}$ of $G_0$, there is some $t\in \mathbb{T}$ such that $t^m=g^m$. Because $\mathbb{T}$ commutes with $g$ (see, for example, \cite{Pan26}*{Lemma 2.13}), we have $(gt^{-1})^m=\mathrm{id}$, that is, $gt^{-1}$ is a torsion element. It is clear that $t\in \mathbb{T}$ can be approximated by torsion elements in $\mathbb{T}$. It follows that $g=(gt^{-1})\cdot t \in K$. This verifies the claim.

Next, we quotient out this $K$ action from $(Y,y,G)$, where $K$ is compact and normal. We consider the quotient metric measure space with the quotient group action
$$(Z,z,H)\coloneqq(Y/K,\hat{y},G/K).$$ 
Thanks to \cite{Galaz-Garcia_18}, the quotient metric measure space $Z$ is $\RCD(0,n)$. From the above claim, $H=G/K$ is homeomorphic to the orbit $Gy$ in $Y$. Because $H$ itself is a nilpotent Lie group, it follows from property (3) that $H$ is a connected and simply connected nilpotent Lie group with $\dim(H)\ge k$. Moreover, by construction and property (6), the orbit $\hat{L}z$ has Hausdorff dimension at least $s\ge 2$, where $\hat{L}\leq H$ is the image of $L$ under the quotient $G\to G/K=H$. Now we have checked that the space $(Z,z,H)$ satisfies all the assumptions in Theorem \ref{mainthm:RCD_N_bdd}. Therefore, we conclude by Theorem \ref{mainthm:RCD_N_bdd} that $n\ge 4s(s-1)+k+1$. This completes the proof.
\end{proof}

\appendix

\section{Convergence of quotient group actions}\label{appx:conv_quotients}
For the sake of completeness, we give a proof of Proposition \ref{prop:eGH_quotient} in this appendix.

We write $(Y,y,G)$ a pointed proper length metric space $(Y,y)$ with an isometric $G$-action, where $G$ is closed in the isometry group of $Y$. Let $R>0$, we denote
$$G(R)=\{ g\in G \mid d(gy,y)\le R \}.$$
For a closed and normal subgroup $H$ of $G$, we can consider the quotient metric space with the quotient isometric action, that is,
$$(Y/H,\bar{y},G/H).$$

\begin{lem}\label{lem:group_proj}
On $(Y,y,H\triangleleft G)$, where $H$ is a closed and normal subgroup of $G$, it holds that
	$$\dfrac{G}{H}(R)=\overline{G(R)}$$
    for all $R>0$, where $\overline{G(R)}$ denotes the image of $G(R)$ under the projection map $G\to G/H$.
\end{lem}

\begin{proof}
   Let $\bar{g}\in \frac{G}{H}(R)$. By the definition of $\frac{G}{H}(R)$, there is some $g\in G$ such that $g$ projects to $\bar{g}$ and $d(\bar{g}\bar{y},\bar{y})\le R$ in $Y/H$. We choose $h_1,h_2\in H$ such that
   $$d(h_1gy,h_2y)=d(Hgy,Hy)=d(\bar{g}\bar{y},\bar{y})\le R.$$
   For $h'=g^{-1}h_2^{-1}h_1g\in H$, we have
   $$R \ge d(h_1gy,h_2y)=d(h_2^{-1}h_1gy,y)=d(gh'y,y).$$
   Then $gh'$ is the desired element in $G(R)$ that projects to $\bar{g}\in G/H$. This proves $\frac{G}{H}(R)\subseteq \overline{G(R)}$.

   For the other direction, if $g\in G(R)$, then in the quotient space $Y/H$, we have
   $$d(\bar{g}\bar{y},\bar{y})=d(Hgy,Hy)\le d(gy,y)\le R.$$
   Hence $\bar{g}\in \frac{G}{H}(R)$.
\end{proof}

\begin{proof}[Proof of Proposition \ref{prop:eGH_quotient}]    
	From the given equivariant Gromov--Hausdorff convergence, we have $\epsilon_i$-approximations
	$$f_i:B_{1/\epsilon_i}(y_i)\to Y,\quad  \psi_i: G_i(1/\epsilon_i) \to G(1/\epsilon_i), \quad \phi_i:G(1/\epsilon_i) \to  G_i(1/\epsilon_i),$$
	where $\epsilon_i\to 0$ (see \cite{FY92}*{Definition 3.3}).  We may also assume that $\psi_i(H_i(1/\epsilon_i))\subseteq H$ and $\phi_i(H(1/\epsilon_i))\subseteq H_i$. According to \cite{Fukaya86}*{Theorem 2.1}, we have the convergence of quotient metric spaces
	$$(Y_i/H_i,\bar{y_i})\overset{GH}\longrightarrow (Y/H,\bar{y});$$
	moreover, we can choose the approximation maps $\bar{f}_i$ from $B_{1/(5\epsilon_i)}(\bar{y}_i)\subseteq Y_i/H_i$ to $Y/H$ as follows: for each $\bar{x}\in B_{1/(5\epsilon_i)}(\bar{y}_i)$, we define
	$$\bar{f}_i(\bar{x})\coloneqq\overline{f_i(x)}\in Y/H,$$
	where $x\in B_{1/(5\epsilon_i)}(y_i)$ is a point projecting to $\bar{x}\in Y_i/H_i$.
	
    We define 
    $$ \bar{\psi}_i: \dfrac{G_i}{H_i}\left(\frac{1}{5\epsilon_i}\right) \to \dfrac{G}{H},\quad \bar{g} \mapsto \overline{\psi_i(g)},$$
	where $g\in G_i(\frac{1}{5\epsilon_i})$ such that $g$ projects to $\bar{g}$; the existence of such an element $g$ is guaranteed by Lemma \ref{lem:group_proj}. With this choice of $g$, $\psi_i(g)$ makes sense. We estimate
	$$d(\overline{\psi_i(g)}\bar{y},\bar{y})=d(H\psi_i(g)y,Hy)\\
		\le d(\psi_i(g)y,y) \le \frac{1}{5\epsilon_i} $$
	Thus $\mathrm{im}(\bar{\psi}_i)\subseteq \frac{G}{H}(\frac{1}{5\epsilon_i})$. For any $\bar{g}\in \frac{G_i}{H_i}(\frac{1}{5\epsilon_i})$ and any $\bar{x}\in B_{1/(5\epsilon_i)}(\bar{y}_i) $ with $\bar{g}\bar{x}\in B_{1/(5\epsilon_i)}(\bar{y}_i)$, 
	\begin{align*}
    d(\bar{f}_i(\bar{g}\bar{x}),\bar{\psi}_i(\bar{g})\bar{f}_i(\bar{x}))& = d(\overline{f_i(gx)},\overline{\psi_i(g)}\cdot \overline{f_i(x)}) \\
		& =d(H\cdot f_i(gx),H\cdot \psi_i(g)\cdot f_i(x))\\
		& \le d(f_i(gx),\psi_i(g)\cdot f_i(x))\\
		& \le \epsilon_i.
	\end{align*}
	Similarly, we can construct 
	$$\bar{\phi}_i: \frac{G}{H}\left(\dfrac{1}{5\epsilon_i}\right) \to \dfrac{G_i}{H_i},\quad \bar{h} \mapsto \overline{\phi_i(h)},$$
	where $h\in G(\frac{1}{5\epsilon_i})$ projects to $\bar{h}$, and show the desired estimates. Therefore, $(\bar{f}_i,\bar{\psi}_i,\bar{\phi}_i)$ gives $(5\epsilon_i)$-approximation maps between the quotient spaces $(Y_i/H_i,\bar{y}_i,G_i/H_i)$ and $(Y/H,\bar{y},G/H)$.
\end{proof}

\section{Perturbation of symmetric matrices}\label{appx:perturbation_matrices}
In this appendix, we prove a trace inequality below. This inequality is used in Section \ref{sec:curv_ineq} to derive the Bakry--Emery Ricci curvature inequality (Proposition \ref{prop:Ric_rr}).

\begin{prop}\label{prop:Trace-inequality}
   Let $I$ be an open interval and let $\{G(r)\}_{r\in I}$ be a continuous and $W^{1,2}_{loc}$ family of $k\times k$ positive definite symmetric matrices. We define
$$S(r)=\dfrac{1}{2} G(r)^{-1} G'(r).$$
   Then
   \begin{equation}\label{eq:trace-inequality}
        \tr(S(r)^2) \ge \dfrac{1}{4}\sum_{j=1}^k \left(\dfrac{\lambda'_j(r)}{\lambda_j(r)}\right)^2
   \end{equation}
   holds for almost every $r\in I$, where $\lambda_1(r)\le \dots \le \lambda_k(r)$ denote the ordered eigenvalues of $G(r)$.
\end{prop}

We review several results about perturbations of matrices with a focus on eigenvalues. Given an open interval $I$ and a family of $k\times k$ matrices $T(r)$ depending on the parameter $r\in I$, assume that the matrices $T(r)$ have eigenvalues $\lambda_1(r),\ldots, \lambda_k(r)$ as complex-valued functions. When $T(r)$ is symmetric for any $r\in I$, then the eigenvalues are real and can be arranged in ascending order as follows
$$\lambda_1(r)\le \ldots \le \lambda_k(r).$$
The perturbation of ordered eigenvalues can be bounded by the perturbation of matrices. The following theorem in fact holds for arbitrary perturbations, but we only state here the version for the one-parameter perturbations.
\begin{thm}[\cite{MatrixAnalysis}*{Corollary 3.2.6}, Weyl's perturbation theorem]\label{thm:Weyl}
For each $1\le j\le k$, the $j$-th ordered eigenvalue is Lipschitz continuous over the space of $k\times k$ symmetric matrices
   $$|\lambda_j(r)-\lambda_j(s)|\le \| T(r)-T(s) \|,$$
   where $\|\cdot\|$ denotes the operator norm.
\end{thm}

 Now we review the differentiation of eigenvalues from the book of Kato \cite{KatoPerturbation}. In general even if $T(r)$ is differentiable at some $r_0\in I$, the eigenvalues of $T(r)$ need not be differentiable at $r_0$. There are two levels of issues.

The first level is that if an eigenvalue $\lambda(r_0)$ of $T(r_0)$ is not semisimple, meaning its geometric multiplicity is not the algebraic multiplicity, then $\lambda(r)$ need not be differentiable. Since we will work with symmetric matrices, this will not be of our concern. 

The next issue arises from ordering. The following elementary example illustrates the point.
\begin{exmp}
    Consider the positive definite matrices 
\[
T(r)=
\begin{pmatrix}
 1&r\\
r&1
\end{pmatrix},
\qquad r\in (-1,1).
\]
The solutions to the characteristic equation are $1-r$, $1+r$. If one arranges the eigenvalues in an ascending order $\lambda_1(r)\le \lambda_2(r)$, then $\lambda_1(r)=1-|r|$, $\lambda_2(r)=1+|r|$, which are not differentiable at $r=0$, even though $T(r)$ is differentiable there. The unordered eigenvalues $\{1-r,1+r\}$ are differentiable at $r=0$. Their derivatives at $0$, which are $\pm 1$, are eigenvalues of $T'(0)$.
\end{exmp}

When discussing the differentiability of eigenvalues, it is more convenient to consider a group of eigenvalues together since the multiplicity of a single eigenvalue may vary as the parameter $r$ varies. We draw the following definition from Kato \cite{KatoPerturbation}. 
\begin{defn}
    Let $\lambda\in \C$, $r_0\in I$ , we say an eigenvalue $\mu_i(r)$ is a $\lambda$-group eigenvalue at $r_0$ if $\mu_i(r)\to \lambda$ as $r\to r_0$. The total projection $P_\lambda(r)$ for each $\lambda$-group at $r_0$ is defined by the contour integral
    \[
    P_{\lambda}(r)\defeq -\frac{1}{2\pi i}\int_{\Gamma}(T(r)-\zeta I)^{-1}\, d\zeta,
    \]
    where $\Gamma$ is a contour that encloses and only encloses the eigenvalues in the chosen $\lambda$-group. In particular, if $T(r_0)$ is symmetric, then $P(r_0)$ is the orthogonal projection onto the eigenspace $E_\lambda\defeq\ker(\lambda I_k-T(r_0))$, where $I_k$ is the $k\times k$ identity matrix.
\end{defn}

The existing literature on the derivatives of eigenvalues is about the unordered ones. We first record the results on the unordered $\lambda$-group eigenvalues. 

\begin{thm}[\cite{KatoPerturbation}*{Chapter Two, Theorem 5.4}]\label{thm:Kato}
    Let $T(r)$ be differentiable at $r_0$, $\lambda\in \C$ a semisimple eigenvalue of $T(r_0)$ with multiplicity $m\in\N$. Then the $\lambda$-group eigenvalues are differentiable at $r_0$ in the following sense. There exist functions $\mu_j(r)$, as unordered $\lambda$-group eigenvalues of $T(r)$ and numbers $\mu_j^{(1)}$, $j=1,\ldots,m$, such that the expansion 
    \begin{equation}\label{eq:eigenvalue-expansion}
        \mu_j(r)=\lambda+\mu_j^{(1)}(r-r_0)+o(|r-r_0|)
    \end{equation}
    holds for $j=1,\ldots, m$, when $r$ is sufficiently close to $r_0$. Moreover, $\mu_1^{(1)},\ldots, \mu_m^{(1)}$ are all the eigenvalues of $P_\lambda(r_0) T'(r_0)P_\lambda(r_0)$ restricted to $E_{\lambda}=\ker(\lambda I_k-T(r_0))$. If in addition $T(r)$ is symmetric for every $r$, then all the unordered eigenvalues are differentiable at $r_0$. 
\end{thm}

However, in Weyl's perturbation theorem and in the trace inequality below we must consider ordered eigenvalues. As can be seen from the previous example, if we only assume the matrix-valued function $T(r)$ is differentiable, then there is no conclusion about the derivatives of ordered eigenvalues. However, if the ordered eigenvalues happen to be also differentiable, we can make an observation that links the derivatives of ordered eigenvalues to that of the unordered ones.

\begin{prop}\label{prop:unordered-derivative=ordered-derivative}
    Let $T(r)$ be a family of symmetric matrices, and $\lambda_1(r)\le\cdots\le \lambda_k(r)$ the ascending ordered eigenvalues of $T(r)$. Assume that $T(r)$ and every $\lambda_i(r)$ are differentiable at $r_0$ as a function of $r$ for $i=1,\ldots, k$, then the derivatives $\lambda_i'(r)$ coincide, counted with multiplicity, with the derivatives of unordered eigenvalues at $r_0$ in the sense of Kato. 
\end{prop}

\begin{proof}
    Let $\mu$ be an eigenvalue of $T(r_0)$ of multiplicity $m$. It suffices to prove the statement for all $\mu$-group eigenvalues. Without loss of generality, we assume $\lambda_{1}(r_0)=\cdots=\lambda_m(r_0)=\mu$. Take $h>0$, for every $j=1,\ldots, m$ the monotonicity gives
    \[
    \frac{\lambda_m(r_0+h)-\mu}{h}\ge \frac{\lambda_j(r_0+h)-\mu}{h}, \quad\frac{\lambda_1(r_0-h)-\mu}{-h}\ge \frac{\lambda_j(r_0-h)-\mu}{-h}.
    \]
  Let $h\to 0^+$, it follows that 
  \[
  \lambda_m'(r_0)=\cdots=\lambda_1'(r_0).
  \]
  Let $\mu_1(r),\ldots,\mu_m(r)$ be a choice of 
$\mu$-group eigenvalues that are differentiable with derivatives $\mu_1^{(1)}\le\ldots\le\mu_m^{(1)} $ at $r_0$. From the ordering we have that for all $r\in I$,
\[
\lambda_m(r)\ge \max\{\mu_1(r),\ldots,\mu_m(r)\}.
\]
Then take $h>0$, we see that  
\[
\frac{\lambda_m(r_0+h)-\mu}{h}\ge\frac{\mu_m(r_0+h)-\mu}{h},\quad \frac{\lambda_m(r_0-h)-\mu}{-h}\le \frac{\mu_1(r_0-h)-\mu}{-h}.
\]
Let $h\to 0^+$, it follows that 
\[
\lambda_m'(r_0)\ge \mu_m^{(1)}\ge \mu_1^{(1)}\ge \lambda_m'(r_0),
\]
which in turn gives  \[
\mu_1^{(1)}=\cdots=\mu_m^{(1)}=\lambda_m'(r_0)=\cdots=\lambda_1'(r_0). \] 
This completes the proof.
\end{proof}

\begin{proof}[Proof of Proposition \ref{prop:Trace-inequality}]
    As a result of Weyl's perturbation Theorem \ref{thm:Weyl}, we infer that $\lambda_j$ lies in $C^0\cap W^{1,2}_{loc}(I)$, and $\lambda_j(r)\ge \delta>0$ is bounded away from $0$ on compact sub-intervals of $I$. Then  $\lambda'_j\in L^2_{loc}$ and $\lambda_j^{-1}\in L^\infty_{loc}$, so the right hand side of \eqref{eq:trace-inequality} is an $L^1_{loc}$ function. Similarly, $S(r)$ is an $L^2_{loc}$ family of matrices, so the left hand side of \eqref{eq:trace-inequality} is also an $L^1_{loc}$ function. So both sides of \eqref{eq:trace-inequality} are well defined.

    We also observe that, being locally $W^{1,2}$ on $I$, $G(r)$ and the ordered eigenvalues are in particular locally absolutely continuous on $I$; thus, they are differentiable almost everywhere on $I$.
    
    Consider the points $r\in I$ where $G(r)$ and $\lambda_j(r)$, $j=1,\ldots,k$ are all differentiable. Fix such an $r_0\in I$. Assume that $G(r_0)$ has $p$ distinct eigenvalues, $\{\bar \lambda_i\}_{i=1}^p$ and each $\bar \lambda_i$ has multiplicity $m_{i}-m_{i-1}$ with $0=m_0<m_1<\ldots<m_p=k$. The ordered eigenvalues can be grouped as
    \[
\underbrace{\lambda_1(r_0)=\cdots=\lambda_{m_1}(r_0)}_{=\bar{\lambda}_1}<\underbrace{\lambda_{m_1+1}(r_0)=\cdots=\lambda_{m_2}(r_0)}_{=\bar{\lambda}_2}<\cdots<\underbrace{\lambda_{m_{p-1}+1}(r_0)=\cdots=\lambda_{m_p}(r_0)}_{=\bar{\lambda}_p}.
    \]
    By Proposition \ref{prop:unordered-derivative=ordered-derivative} it follows every $\bar \lambda_i$-group has derivative 
    \begin{equation}\label{eq:equal_derivatives_in_lambda_group}
        \bar\lambda_i^{(1)}\defeq \lambda_{m_{i-1}+1}'(r_0)=\cdots=\lambda_{m_i}'(r_0).
    \end{equation}
    Let $E_i\coloneqq \ker(\bar{\lambda}_{i}I_k-G(r_0))$ and let $P_i$ be the matrix of the orthogonal projection onto $E_i$ in the canonical basis $\mathcal{B}_{can}$ of $\R^k$. We fix an orthonormal basis $\mathcal{B}_i$ of $E_i$ and let $Q$ be the matrix of change of basis from $\mathcal{B}_{can}$ to $\mathcal{B}=(\mathcal{B}_1,\ldots,\mathcal{B}_p)$. In particular, $Q^TG(r_0)Q$ is diagonal with diagonal blocks $\bar{\lambda}_{1}I_{m_1},\ldots, \bar{\lambda}_{p}I_{m_p-m_{p-1}}$. Note that $Q^TG'(r_0)Q$ is the matrix of $G'(r_0)$ (seen as an endomorphism of $\R^k$) in the basis $\mathcal{B}$. Thus, its diagonal blocks coincide with the matrices $M_1,\ldots, M_p$, where $M_i$ is the matrix of $P_iG'(r_0)_{\lvert E_i}$ (seen as an endomorphism of $E_i$) in the basis $\mathcal{B}_i$. However, thanks to Theorem \ref{thm:Kato} and \eqref{eq:equal_derivatives_in_lambda_group}, $P_iG'(r_0)_{\lvert E_i}=P_iG'(r_0){P_i}_{\lvert E_i}=\bar\lambda_i^{(1)}\mathrm{id}_{E_i}$. Therefore, while $Q^T G'(r_0)Q$ is not necessarily diagonal, its diagonal blocks are $\bar{\lambda}^{(1)}_{1}I_{m_1},\ldots, \bar{\lambda}^{(1)}_{p}I_{m_p-m_{p-1}}$. In particular, the diagonal coefficients of
    \[
    H\defeq Q^T G(r_0)^{-1/2}G'(r_0)G^{-1/2}(r_0) Q,
    \]
    satisfy $H_{jj}=\lambda_j'(r_0)/\lambda_j(r_0)$. Since $H$ is similar to $2S(r_0)=G(r_0)^{-1}G'(r_0)$, we have that
    \[
    \tr(S(r_0)^2)=\frac{1}{4}\tr(H(r_0)^2)=\frac{1}{4}\sum_{1\le i,j\le k} H_{ij}^2\ge \frac{1}{4}\sum_{j=1}^k H_{jj}^2=\frac{1}{4}\sum_{j=1}^k\left(\dfrac{\lambda'_j(r_0)}{\lambda_j(r_0)}\right)^2,
    \]
    as desired.
\end{proof}

\bibliographystyle{plain}
\bibliography{ref.bib}

\end{document}